\documentclass[
final
 , nomarks
]{dmtcs-episciences}

\usepackage[utf8]{inputenc}
\usepackage{subfigure}

\usepackage[round]{natbib}

\usepackage{amsmath}
\usepackage{amssymb}
\usepackage{amsthm}
\usepackage{enumerate}
\usepackage[svgnames]{xcolor}
\usepackage{tikz}
\usepackage{ifthen}
\usepackage{tabu}
\usepackage{longtable}
\newcommand{\N}{\mathbb{N}}
\newcommand{\Z}{\mathbb{Z}}

\newcommand{\Zn}[1]{\Z/{#1}\Z}
\renewcommand{\ge}{\geqslant}

\renewcommand{\le}{\leqslant}

\newcommand{\ST}{\nabla}
\newcommand{\STset}[1]{\mathcal{ST}_{#1}} 
\newcommand{\PT}{\Delta}
\newcommand{\PTset}[1]{\mathcal{PT}_{#1}} 
\newcommand{\orb}[1]{\mathcal{O}_{#1}}
\newcommand{\PO}[1]{\mathcal{PO}_{#1}}
\newcommand{\BPO}[1]{\mathcal{BPO}_{#1}}
\newcommand{\resPO}[1]{\res{\PO{#1}}}
\newcommand{\BresPO}[1]{\res{\BPO{#1}}}
\newcommand{\remaind}[1]{\mathcal{R}_{#1}}
\newcommand{\m}{\mathfrak{m}}
\newcommand{\res}[1]{\overline{#1}}
\newcommand{\ls}{\mathsf{l}}
\newcommand{\rs}{\mathsf{r}}
\newcommand{\is}{\mathsf{i}}
\newcommand{\ts}[2]{\mathsf{t}_{#1,#2}}
\newcommand{\ord}[2]{\mathrm{ord_{#1}\left(#2\right)}}
\newcommand{\clr}{cyan}
\newcommand{\clrF}{blue}
\newcommand{\clrd}{magenta}
\newcommand{\clrdF}{red}
\newcommand{\clrt}{yellow}

\newcommand{\side}{1}
\newcounter{i}
\newcounter{j}

\newcommand{\element}[3]{
	\pgfmathparse{#1}\let\a\pgfmathresult
	\pgfmathparse{#1+\side}\let\b\pgfmathresult
	\pgfmathparse{#2-\side}\let\c\pgfmathresult
	\pgfmathparse{#2}\let\d\pgfmathresult
	\pgfmathparse{#1+\side/2}\let\e\pgfmathresult
	\pgfmathparse{#2-\side/2}\let\f\pgfmathresult
	\draw (\a,\c) -- (\a,\d) -- (\b,\d) -- (\b,\c) -- (\a,\c);
	\draw (\e,\f) node {#3};
}
\newcommand{\celement}[3]{
	\pgfmathparse{#1}\let\a\pgfmathresult
	\pgfmathparse{#1+\side}\let\b\pgfmathresult
	\pgfmathparse{#2-\side}\let\c\pgfmathresult
	\pgfmathparse{#2}\let\d\pgfmathresult
	\pgfmathparse{#1+\side/2}\let\e\pgfmathresult
	\pgfmathparse{#2-\side/2}\let\f\pgfmathresult
	\draw[fill=#3] (\a,\c) -- (\a,\d) -- (\b,\d) -- (\b,\c) -- (\a,\c);
}
\newcommand{\ligne}[3]{
	\setcounter{j}{0}
	\foreach \x in #3 {
		\pgfmathparse{#1+\value{j}*\side}\let\la\pgfmathresult
		\element{\la}{#2}{\x}
		\addtocounter{j}{1}
	}
}
\newcommand{\cligne}[3]{
	\setcounter{j}{0}
	\foreach \x in #3 {
		\pgfmathparse{#1+\value{j}*\side}\let\la\pgfmathresult
		\celement{\la}{#2}{\x}
		\addtocounter{j}{1}
	}
}
\newcommand{\grille}[3]{
	\setcounter{i}{0}
	\foreach \x in #3 {
		\pgfmathparse{#2-\value{i}*\side}\let\ca\pgfmathresult
		\ligne{#1}{\ca}{\x}
		\addtocounter{i}{1}
	}	
}
\newcommand{\cgrille}[3]{
	\setcounter{i}{0}
	\foreach \x in #3 {
		\pgfmathparse{#2-\value{i}*\side}\let\ca\pgfmathresult
		\cligne{#1}{\ca}{\x}
		\addtocounter{i}{1}
	}	
}
\newcommand{\carre}[3]{
	\draw (#1,#2) -- (#1+1,#2) -- (#1+1,#2-1) -- (#1,#2-1) -- (#1,#2);
	\draw[fill=black] (#1+#3,#2-#3) -- (#1+1-#3,#2-#3) -- (#1+1-#3,#2-1+#3) -- (#1+#3,#2-1+#3) -- (#1+#3,#2-#3);
}
\newcommand{\carreL}[4]{
	\setcounter{i}{0}
	\foreach \x in #4{
		\ifthenelse{\equal{\x}{0}}{\carre{#1+\value{i}}{#2}{0.5}}{\carre{#1+\value{i}}{#2}{#3}}
		\addtocounter{i}{1}
	}
}

\theoremstyle{plain}
\newtheorem{thm}{Theorem}[section]
\newtheorem{lem}[thm]{Lemma}
\newtheorem{prop}[thm]{Proposition}

\theoremstyle{definition}

\newtheorem{prob}{Problem}
\newtheorem*{SteinProb}{Steinhaus Problem}
\newtheorem*{MollProb}{Molluzzo Problem}
\theoremstyle{remark}

\author{Jonathan Chappelon}
\title{Periodic balanced binary triangles}
\affiliation{
  IMAG, Univ Montpellier, CNRS, Montpellier, France}
\received{2017-2-14}
\revised{2017-11-6}
\accepted{2017-11-13}
\publicationdetails{19}{2017}{3}{13}{3141}
\begin{document}
\maketitle
\begin{abstract}
A binary triangle of size $n$ is a triangle of zeroes and ones, with $n$ rows, built with the same local rule as the standard Pascal triangle modulo $2$. A binary triangle is said to be balanced if the absolute difference between the numbers of zeroes and ones that constitute this triangle is at most $1$. In this paper, the existence of balanced binary triangles of size $n$, for all positive integers $n$, is shown. This is achieved by considering periodic balanced binary triangles, that are balanced binary triangles where each row, column or diagonal is a periodic sequence.
\end{abstract}

\section{Introduction}\label{sec:1}

The {\em Steinhaus triangle} $\ST{S}$ associated with the finite sequence $S=(a_0,a_1,\ldots,a_{n-1})$, of length $n\ge1$ in $\Zn{2}$, is the triangle generated from $S$ by the same local rule that defines the standard Pascal triangle modulo $2$, that is the doubly indexed sequence $\ST{S}=\left(a_{i,j}\right)_{0\le i\le j\le n-1}$ defined by:
\begin{enumerate}[i)]
\item
$a_{0,j} = a_{j}$, for all $0\le j\le n-1$,
\item\label{Prule}
$a_{i,j} = a_{i-1,j-1} + a_{i-1,j}$, for all $1\le i\le j\le n-1$.
\end{enumerate}
Note that the sum in \ref{Prule}) is the sum modulo $2$. The $m$-th {\em row}, {\em column} and {\em diagonal} of the Steinhaus triangle $\ST{S}$ are the sequences $(a_{m,j})_{m\le j\le n-1}$, $(a_{i,m})_{0\le i\le m}$ and $(a_{i,m+i})_{0\le i\le n-1-m}$, respectively, for all $m\in\{0,1,\ldots,n-1\}$. This kind of binary triangle was introduced in \citet{Steinhaus:1964aa}. For example, the Steinhaus triangle $\ST{S}$ associated with $S=0010100$ is depicted in Figure~\ref{fig1}.

The {\em generalized Pascal triangle} $\PT{(S_l,S_r)}$ associated with the finite sequences $S_l=(a_0,a_1,\ldots,a_{n-1})$ and $S_r=(b_0,b_1,\ldots,b_{n-1})$, of length $n\ge1$ in $\Zn{2}$ and with $a_0=b_0$, is the doubly indexed sequence $\PT{(S_l,S_r)}=\left(a_{i,j}\right)_{0\le j \le i \le n-1}$ defined by:
\begin{enumerate}[i)]
\item
$a_{i,0} = a_{i}$ and $a_{i,i} = b_{i}$, for all $0\le i\le n-1$,
\item
$a_{i,j} = a_{i-1,j-1} + a_{i-1,j}$, for all $1\le j < i\le n-1$.
\end{enumerate}
The $m$-th {\em row}, {\em column} and {\em diagonal} of the generalized Pascal triangle $\PT{(S_l,S_r)}$ are the sequences $(a_{m,j})_{0\le j\le m}$, $(a_{i,m})_{m\le i\le n-1}$ and $(a_{m+j,j})_{0\le j\le n-1-m}$, respectively, for all $m\in\{0,1,\ldots,n-1\}$. For example, the generalized Pascal triangle $\PT(S_l,S_r)$ associated with $S_l=0000101$ and $S_r=0100001$ is depicted in Figure~\ref{fig1}. Moreover, note that, for the constant binary sequences $S_l=S_r=11\cdots1$ of size $n$, the triangle $\PT(S_l,S_r)$ corresponds to the first $n$ rows of the standard Pascal triangle modulo $2$, the Sierpinski triangle.

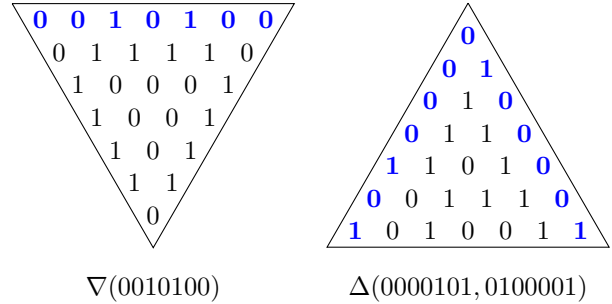
\begin{figure}[htbp]
\begin{center}
\begin{tabular}{cc}
\begin{tikzpicture}[scale=0.25]
\pgfmathparse{sqrt(3)}\let\x\pgfmathresult

\node at (0,0) {${\bf \color{blue}{0}}$};
\node at (2,0) {${\bf \color{blue}{0}}$};
\node at (4,0) {${\bf \color{blue}{1}}$};
\node at (6,0) {${\bf \color{blue}{0}}$};
\node at (8,0) {${\bf \color{blue}{1}}$};
\node at (10,0) {${\bf \color{blue}{0}}$};
\node at (12,0) {${\bf \color{blue}{0}}$};

\node at (1,-\x) {$0$};
\node at (3,-\x) {$1$};
\node at (5,-\x) {$1$};
\node at (7,-\x) {$1$};
\node at (9,-\x) {$1$};
\node at (11,-\x) {$0$};

\node at (2,-2*\x) {$1$};
\node at (4,-2*\x) {$0$};
\node at (6,-2*\x) {$0$};
\node at (8,-2*\x) {$0$};
\node at (10,-2*\x) {$1$};

\node at (3,-3*\x) {$1$};
\node at (5,-3*\x) {$0$};
\node at (7,-3*\x) {$0$};
\node at (9,-3*\x) {$1$};

\node at (4,-4*\x) {$1$};
\node at (6,-4*\x) {$0$};
\node at (8,-4*\x) {$1$};

\node at (5,-5*\x) {$1$};
\node at (7,-5*\x) {$1$};

\node at (6,-6*\x) {$0$};

\draw (-1.5,0.5*\x) -- (13.5,0.5*\x) -- (6,-7*\x) -- (-1.5,0.5*\x);
\end{tikzpicture}
&
\begin{tikzpicture}[scale=0.25]
\pgfmathparse{sqrt(3)}\let\x\pgfmathresult

\node at (0,0) {${\bf \color{blue}{0}}$};

\node at (-1,-\x) {${\bf \color{blue}{0}}$};
\node at (1,-\x) {${\bf \color{blue}{1}}$};

\node at (-2,-2*\x) {${\bf \color{blue}{0}}$};
\node at (0,-2*\x) {$1$};
\node at (2,-2*\x) {${\bf \color{blue}{0}}$};

\node at (-3,-3*\x) {${\bf \color{blue}{0}}$};
\node at (-1,-3*\x) {$1$};
\node at (1,-3*\x) {$1$};
\node at (3,-3*\x) {${\bf \color{blue}{0}}$};

\node at (-4,-4*\x) {${\bf \color{blue}{1}}$};
\node at (-2,-4*\x) {$1$};
\node at (0,-4*\x) {$0$};
\node at (2,-4*\x) {$1$};
\node at (4,-4*\x) {${\bf \color{blue}{0}}$};

\node at (-5,-5*\x) {${\bf \color{blue}{0}}$};
\node at (-3,-5*\x) {$0$};
\node at (-1,-5*\x) {$1$};
\node at (1,-5*\x) {$1$};
\node at (3,-5*\x) {$1$};
\node at (5,-5*\x) {${\bf \color{blue}{0}}$};

\node at (-6,-6*\x) {${\bf \color{blue}{1}}$};
\node at (-4,-6*\x) {$0$};
\node at (-2,-6*\x) {$1$};
\node at (0,-6*\x) {$0$};
\node at (2,-6*\x) {$0$};
\node at (4,-6*\x) {$1$};
\node at (6,-6*\x) {${\bf \color{blue}{1}}$};

\draw (0,\x) -- (-7.5,-6.5*\x) -- (7.5,-6.5*\x) -- (0,\x);
\end{tikzpicture} \\[1.25ex]
$\ST(0010100)$ & $\PT(0000101,0100001)$ \\
\end{tabular}
\caption{Binary triangles}\label{fig1}
\end{center}
\end{figure}

In this paper, a {\em binary triangle} is either a Steinhaus triangle or a generalized Pascal triangle. The {\em size} of a binary triangle is the number of rows that constitute this triangle.

For any binary triangle $T$, let $\m_T$ denote its multiplicity function, that is, the function $\m_T : \Zn{2} \longrightarrow \N$ that assigns to each element $x\in\Zn{2}$ its multiplicity in $T$. The triangle $T$ is said to be {\em balanced} if its multiplicity function is constant or almost constant, i.e., if the multiplicity difference
$$
\delta\m_T:=|\m_T(1)-\m_T(0)|
$$
is such that $\delta\m_T\in\{0,1\}$. Since they contain $14$ zeroes and $14$ ones, the triangles depicted in Figure~\ref{fig1} are balanced binary triangles of size $7$.

The goal of this paper is to prove that there exist balanced binary triangles of size $n$, for all positive integers $n$ and for the both kinds of binary triangles. This completely solves a generalization of a problem posed in \citet{Steinhaus:1964aa}.

\begin{SteinProb}
Does there exist, for any positive integer $n\equiv 0$ or $3\bmod{4}$, a binary sequence $S$ of length $n$ for which the associated triangle $\ST{S}$ contains as many zeroes as ones?
\end{SteinProb}

Since a binary triangle of size $n$ contains ${{n+1}\choose{2}}$ elements, the condition $n\equiv 0$ or $3\bmod{4}$ is a necessary and sufficient condition for having a triangle of size $n$ containing an even number of terms.

The Steinhaus Problem was solved for the first time in \citet{Harborth:1972aa}. In his paper, Harborth constructively showed that, for every positive integer $n\equiv 0$ or $3\bmod{4}$, there exist at least four binary sequences $S$ of length $n$ such that $\ST S$ is balanced. Since then, many solutions have appeared \citep*{Eliahou:2004aa,Eliahou:2005aa,Eliahou:2007aa}. All of them are constructive and correspond to the search of sequences generating balanced triangles, that have some additional properties such as being antisymmetric or zero-sum.

The possible number of ones in binary triangles was explored in \citet{Chang:1983aa,Harborth:2005aa}. The minimum number of ones is obviously $0$ since the triangle of zeroes of size $n$ is always a binary triangle. The maximum number of ones in a Steinhaus triangle of size $n$ is $\left\lceil\frac{2}{3}{n+1\choose 2}\right\rceil$. As shown in \citet{Chang:1983aa,Harborth:1972aa}, this maximum number of ones is obtained, for instance, for the Steinhaus triangle associated with the initial segment of length $n$ of the $3$-periodic sequence $110110110\cdots$ for all positive integer $n$. In \citet{Harborth:2005aa}, it was proved that the maximum number of ones in a generalized Pascal triangle of size $n$ is $\left\lceil\frac{2}{3}{n+1\choose 2}\right\rceil + \varepsilon$, where
$$
\varepsilon = \left\{\begin{array}{ll}
2 & \text{if } n\equiv1\bmod{3}, n\neq1, \\
1 & \text{if } n\equiv0,2\bmod{3}, n\neq8, \\
0 & \text{if } n=1,\\
3 & \text{if } n=8.
\end{array}\right.
$$
The following result gives the average number of ones and zeroes in binary triangles.

\begin{prop}
The average number of ones and zeroes in a binary triangle of size $n$ is exactly $\frac{1}{2}{n+1\choose 2}$.
\end{prop}

\begin{proof}
We induct on $n$.

First, for the Steinhaus triangles, the result is trivial for $n=1$. Suppose now that $n\ge 2$ and that the result is true for any Steinhaus triangle of size $n-1$. Let $\ST{S}$ be the Steinhaus triangle of size $n-1$ generated from the sequence $S=(a_0,a_1,\ldots,a_{n-2})$. There exist exactly two sequences $S'$ of length $n$ such that we retrieve $\ST{S}$ as the subtriangle $\ST{S'}\setminus S'$, that is, the last $n-1$ rows of the Steinhaus triangle $\ST{S'}$ of size $n$. These sequences $S'$ are of the form
$$
S' = \left( x , x+a_0 , x+a_0+a_1 , \ldots , x+\sum_{j=0}^{i-1}a_j , \ldots , x+\sum_{j=0}^{n-2}a_j \right),
$$
with $x\in\Zn{2}$. Moreover, for all positive integers $m$, it is clear that there are $2^m$ binary sequences of length $m$ and the same number of Steinhaus triangles of size $m$. It follows that, for all $x\in\Zn{2}$, the total number of $x$ in the set of all the Steinhaus triangles of size $n$ is the sum of twice the total number of $x$ in the set of Steinhaus triangles of size $n-1$ and the total number of $x$ in the set of sequences of length $n$. This leads to the result that the average number of $x$ in a Steinhaus triangle of size $n$ is
$$
\frac{1}{2^n} \left( 2\times 2^{n-1}\times \frac{1}{2}{n\choose 2} + 2^{n-1}n \right) = \frac{1}{2}{{n+1}\choose 2},
$$
for all $x\in\Zn{2}$.

Now, for the generalized Pascal triangles. For $n=1$ and $n=2$, the result is clear. Suppose now that $n\ge 3$ and that the result is true for any generalized Pascal triangle of size $n-2$. Let $\PT(S_l,S_r)$ be the generalized Pascal triangle of size $n-2$ generated from the sequences $S_l=(a_0,a_1,\ldots,a_{n-3})$ and $S_r=(b_0,b_1,\ldots,b_{n-3})$, with $a_0=b_0$. There exist exactly $2^4$ couples $(S_l',S_r')$ of sequences of length $n$ such that we retrieve $\PT(S_l,S_r)$ as the subtriangle $\PT(S_l',S_r')\setminus (S_l'\cup S_r'\setminus\{a_0\})$, that is, the generalized Pascal triangle obtained from $\PT(S_l',S_r')$ by removing the left and right sides of the triangle. These couples of sequences $(S_l',S_r')$ are of the form
$$
\begin{array}{l}
S_l' = \left( x_1 , x_2 , a_0+a_1 , a_1+a_2 , \ldots , a_{n-4}+a_{n-3} , x_3 \right), \\[1.5ex]
S_r' = \left( x_1 , x_2+a_0 , b_0+b_1 , b_1+b_2 , \ldots , b_{n-4}+b_{n-3} , x_4 \right), \\
\end{array}
$$
where $x_1,x_2,x_3,x_4\in\Zn{2}$. Moreover, for all positive integers $m$, it is clear that there are $2^{2m-1}$ binary sequences of length $2m-1$ and the same number of generalized Pascal triangles of size $m$. It follows that, for all $x\in\Zn{2}$, the total number of $x$ in the set of all the generalized Pascal triangles of size $n$ is the sum of $2^4$ times the total number of $x$ in the set of generalized Pascal triangles of size $n-2$ and the total number of $x$ in the set of sequences of length $2n-1$. This leads to the result that the average number of $x$ in a generalized Pascal triangle of size $n$ is
$$
\frac{1}{2^{2n-1}} \left( 2^4\times 2^{2(n-2)-1}\times \frac{1}{2}{n-1\choose 2} + 2^{2n-2}(2n-1) \right) = \frac{1}{2}{{n+1}\choose 2},
$$
for all $x\in\Zn{2}$. This completes the proof.
\end{proof}

This result shows that the Steinhaus Problem and the following generalization are {\em natural}.

\begin{prob}\label{prob1}
Does there exist, for any positive integer $n$, a balanced Steinhaus triangle and a balanced generalized Pascal triangle of size $n$?
\end{prob}

As already announced, this problem is solved in the positive in this paper. The solution presented here is constructive and based on the analysis of periodic balanced binary triangles, that are balanced binary triangles where each row or column is a periodic sequence.

Let us begin with some definitions and terminology. Let $S=\left(a_{j}\right)_{j\in\Z}$ be a doubly infinite sequence of $\Zn{2}$. The {\em derived sequence} $\partial S$ is the sequence obtained by pairwise adding consecutive terms of $S$, that is,
$$
\partial S = \left( a_{j-1}+a_{j} \right)_{j\in\Z}.
$$
This derivation process can be iterated and, for every positive integer $i$, the $i$-th derived sequence $\partial^{i}S$ is recursively defined by $\partial^{i}=\partial\left(\partial^{i-1}S\right)$ with $\partial^{0}S = S$. The {\em orbit} $\orb{S}$ is the sequence of all the iterated derived sequences of $S$, that is,
$$
\orb{S} = \left(\partial^{i}S\right)_{i\in\N}.
$$
The orbit of $S$ can also be seen as the doubly indexed sequence $\orb{S}=\left(a_{i,j}\right)_{(i,j)\in\N\times\Z}$ defined by
\begin{enumerate}[i)]
\item
$a_{0,j} = a_{j}$, for all $j\in\Z$, and
\item
$a_{i,j} = a_{i-1,j-1} + a_{i-1,j}$, for all $i\ge 1$ and for all $j\in\Z$.
\end{enumerate}
An example of orbit $\orb{S}$ associated with the sequence
$$
S=\left(\ldots\ldots,0,0,1,0,1,0,1,1,0,0,0,0,1,1,0,0,0,0,1,0,0,1,1,1,0,\ldots\ldots\right)
$$
is depicted in Figure~\ref{fig2}.

\begin{figure}[htbp]
\begin{center}
\begin{tikzpicture}[scale=0.45]
\draw[fill=\clr] (4,-4) -- (4,-5) -- (5,-5) -- (5,-6) -- (6,-6) -- (6,-7) -- (7,-7) -- (7,-8) -- (8,-8) -- (8,-9) -- (9,-9) -- (9,-10) -- (10,-10) -- (10,-4) -- (4,-4);
\draw[fill=\clr] (13,-4) -- (14,-4) -- (14,-5) -- (15,-5) -- (15,-6) -- (16,-6) -- (16,-7) -- (17,-7) -- (17,-8) -- (18,-8) -- (18,-9) -- (19,-9) -- (19,-10) -- (13,-10) -- (13,-4);
\grille{-\side}{0}{{{0,0, 1, 0, 1, 0, 1, 1, 0, 0, 0, 0, 1, 1, 0, 0, 0, 0, 1, 0, 0, 1, 1, 1, 0},
 {1,0, 1, 1, 1, 1, 1, 0, 1, 0, 0, 0, 1, 0, 1, 0, 0, 0, 1, 1, 0, 1, 0, 0, 1},
 {1,1, 1, 0, 0, 0, 0, 1, 1, 1, 0, 0, 1, 1, 1, 1, 0, 0, 1, 0, 1, 1, 1, 0, 1},
 {1,0, 0, 1, 0, 0, 0, 1, 0, 0, 1, 0, 1, 0, 0, 0, 1, 0, 1, 1, 1, 0, 0, 1, 1},
 {0,1, 0, 1, 1, 0, 0, 1, 1, 0, 1, 1, 1, 1, 0, 0, 1, 1, 1, 0, 0, 1, 0, 1, 0},
 {1,1, 1, 1, 0, 1, 0, 1, 0, 1, 1, 0, 0, 0, 1, 0, 1, 0, 0, 1, 0, 1, 1, 1, 1},
 {0,0, 0, 0, 1, 1, 1, 1, 1, 1, 0, 1, 0, 0, 1, 1, 1, 1, 0, 1, 1, 1, 0, 0, 0},
 {0,0, 0, 0, 1, 0, 0, 0, 0, 0, 1, 1, 1, 0, 1, 0, 0, 0, 1, 1, 0, 0, 1, 0, 0},
 {0,0, 0, 0, 1, 1, 0, 0, 0, 0, 1, 0, 0, 1, 1, 1, 0, 0, 1, 0, 1, 0, 1, 1, 0},
 {1,0, 0, 0, 1, 0, 1, 0, 0, 0, 1, 1, 0, 1, 0, 0, 1, 0, 1, 1, 1, 1, 1, 0, 1},
 {1,1, 0, 0, 1, 1, 1, 1, 0, 0, 1, 0, 1, 1, 1, 0, 1, 1, 1, 0, 0, 0, 0, 1, 1},
 {0,0, 1, 0, 1, 0, 0, 0, 1, 0, 1, 1, 1, 0, 0, 1, 1, 0, 0, 1, 0, 0, 0, 1, 0},
 {1,0, 1, 1, 1, 1, 0, 0, 1, 1, 1, 0, 0, 1, 0, 1, 0, 1, 0, 1, 1, 0, 0, 1, 1},
 {0,1, 1, 0, 0, 0, 1, 0, 1, 0, 0, 1, 0, 1, 1, 1, 1, 1, 1, 1, 0, 1, 0, 1, 0}}}
\end{tikzpicture}
\caption{Binary triangles appearing in an orbit $\orb{S}$}\label{fig2}
\end{center}
\end{figure}

Binary triangles can then been considered as appearing in orbits of binary sequences. Let $\ST{S}(i_0,j_0,n)$ denote the triangle build from the base to the top, whose {\em principal vertex} is at the position $(i_0,j_0)\in\N\times\Z$ in the orbit $\orb{S}=\left(a_{i,j}\right)_{(i,j)\in\N\times\Z}$ and of {\em size} $n$, i.e., the Steinhaus triangle
$$
\ST{S}(i_0,j_0,n) = \ST{(a_{i_0,j_0},a_{i_0,j_0+1},\ldots,a_{i_0,j_0+n-1})} = \left( a_{i_0+i,j_0+j} \right)_{0\le i\le j\le n-1}.
$$
Let $\PT{S}(i_0,j_0,n)$ denote the triangle build from the top to the base, whose {\em principal vertex} is at the position $(i_0,j_0)\in\N\times\Z$ in the orbit $\orb{S}$ and of {\em size} $n$, i.e., the generalized Pascal triangle
$$
\PT{S}(i_0,j_0,n) \begin{array}[t]{l}
 = \displaystyle\PT((a_{i_0,j_0},a_{i_0+1,j_0},\ldots,a_{i_0+n-1,j_0}),(a_{i_0,j_0},a_{i_0+1,j_0+1},\ldots,a_{i_0+n-1,j_0+n-1})) \\[1.5ex]
 =  \displaystyle\left( a_{i_0+i,j_0+j} \right)_{0\le j\le i\le n-1}. \\
\end{array}
$$
Example of triangles appearing in an orbit $\orb{S}$ is represented in Figure~\ref{fig2}.

The sets of all the Steinhaus triangles of size $n$ and of all the generalized Pascal triangles of size $n$ are denoted by $\STset{n}$ and $\PTset{n}$, respectively. It is clear that these sets are $\Zn{2}$-vector spaces of dimension $n$ and $2n-1$, respectively. Moreover, as depicted in Figure~\ref{fig11}, there exists an obvious isomorphism between $\PTset{n}$ and $\STset{2n-1}$ since a generalized Pascal triangle of size $n$ can be seen as the center of a Steinhaus triangle of size $2n-1$.

\begin{figure}[htbp]
\begin{center}
\begin{tabular}{ccc}
\begin{tikzpicture}[scale=0.25]
\pgfmathparse{sqrt(3)}\let\x\pgfmathresult

\node at (0,0) {$0$};
\node at (2,0) {$0$};
\node at (4,0) {$1$};
\node at (6,0) {$0$};
\node at (8,0) {$0$};
\node at (10,0) {$1$};
\node at (12,0) {$1$};

\node at (1,-\x) {$0$};
\node at (3,-\x) {$1$};
\node at (5,-\x) {$1$};
\node at (7,-\x) {$0$};
\node at (9,-\x) {$1$};
\node at (11,-\x) {$0$};

\node at (2,-2*\x) {$1$};
\node at (4,-2*\x) {$0$};
\node at (6,-2*\x) {$1$};
\node at (8,-2*\x) {$1$};
\node at (10,-2*\x) {$1$};

\node at (3,-3*\x) {$1$};
\node at (5,-3*\x) {$1$};
\node at (7,-3*\x) {$0$};
\node at (9,-3*\x) {$0$};

\node at (4,-4*\x) {$0$};
\node at (6,-4*\x) {$1$};
\node at (8,-4*\x) {$0$};

\node at (5,-5*\x) {$1$};
\node at (7,-5*\x) {$1$};

\node at (6,-6*\x) {$0$};

\draw (-1.5,0.5*\x) -- (13.5,0.5*\x) -- (6,-7*\x) -- (-1.5,0.5*\x);
\draw (6,\x) -- (1.5,-3.5*\x) -- (10.5,-3.5*\x) -- (6,\x);
\end{tikzpicture}
& &
\begin{tikzpicture}[scale=0.45]
\draw[fill=\clr] (3,0) -- (4,0) -- (4,-1) -- (5,-1) -- (5,-2) -- (6,-2) -- (6,-3) -- (7,-3) -- (7,-4) -- (3,-4) -- (3,0);
\draw (0,0) -- (0,-1);
\draw (1,0) -- (1,-2);
\draw (2,0) -- (2,-3);
\draw (3,0) -- (3,-4);
\draw (4,0) -- (4,-5);
\draw (5,0) -- (5,-6);
\draw (6,0) -- (6,-7);
\draw (7,0) -- (7,-7);
\draw (0,0) -- (7,0);
\draw (0,-1) -- (7,-1);
\draw (1,-2) -- (7,-2);
\draw (2,-3) -- (7,-3);
\draw (3,-4) -- (7,-4);
\draw (4,-5) -- (7,-5);
\draw (5,-6) -- (7,-6);
\draw (6,-7) -- (7,-7);

\node at (0.5,-0.5) {$0$};
\node at (1.5,-0.5) {$0$};
\node at (2.5,-0.5) {$1$};
\node at (3.5,-0.5) {$0$};
\node at (4.5,-0.5) {$0$};
\node at (5.5,-0.5) {$1$};
\node at (6.5,-0.5) {$1$};

\node at (1.5,-1.5) {$0$};
\node at (2.5,-1.5) {$1$};
\node at (3.5,-1.5) {$1$};
\node at (4.5,-1.5) {$0$};
\node at (5.5,-1.5) {$1$};
\node at (6.5,-1.5) {$0$};

\node at (2.5,-2.5) {$1$};
\node at (3.5,-2.5) {$0$};
\node at (4.5,-2.5) {$1$};
\node at (5.5,-2.5) {$1$};
\node at (6.5,-2.5) {$1$};

\node at (3.5,-3.5) {$1$};
\node at (4.5,-3.5) {$1$};
\node at (5.5,-3.5) {$0$};
\node at (6.5,-3.5) {$0$};

\node at (4.5,-4.5) {$0$};
\node at (5.5,-4.5) {$1$};
\node at (6.5,-4.5) {$0$};

\node at (5.5,-5.5) {$1$};
\node at (6.5,-5.5) {$1$};

\node at (6.5,-6.5) {$0$};
\end{tikzpicture}
\end{tabular}
\caption{Isomorphism between $\PTset{n}$ and $\STset{2n-1}$}\label{fig11}
\end{center}
\end{figure}

A binary triangle of size $n$ is constituted by ${{n+1}\choose 2}$ elements, the $n$-th triangular number. Therefore, a binary triangle of size $n$ contains an even number of terms for $n\equiv0,3\bmod{4}$ and an odd number of terms for $n\equiv1,2\bmod{4}$. It follows that a binary triangle $T$ of size $n$ is balanced if and only if
$$
\delta\m_T = \left\{\begin{array}{ll}
 0 & \text{for }n\equiv0,3\bmod{4}, \\
 1 & \text{for }n\equiv1,2\bmod{4}.
\end{array}\right.
$$
In other words, a binary triangle $T$ of size $n$ is balanced if and only if either $\m_T(0)=\m_T(1)$, for $n\equiv0,3\bmod{4}$, or $\m_T(0)=\m_T(1)\pm1$, for $n\equiv1,2\bmod{4}$. Then, the Steinhaus Problem is solved if we can determine whether there exist balanced Steinhaus triangles containing an even number of terms, for all the admissible sizes.

The main result of this paper is the following

\begin{thm}\label{thm1}
There exists a binary doubly infinite sequence $S$ such that its orbit $\orb{S}$ contains balanced Steinhaus triangles and balanced generalized Pascal triangles of size $n$, for all positive integers $n$.
\end{thm}

This theorem completely and positively solves Problem~\ref{prob1}, the generalization of the Steinhaus Problem for the two kinds of triangles, even when the triangles contain an odd number of terms. Note that the existence of balanced Steinhaus triangles with odd cardinality was first announced, without proof, in \citet{Eliahou:2004aa}. For the generalized Pascal triangles, the result is known but has not been published.

This paper is organized as follows. In the next section, the behavior of the $p$-periodic sequences under the action of the derivation process is studied and the set of $p$-tuples that generate $p$-periodic orbits is determined, for all values of $p$. An equivalence relation on the set of $p$-periodic orbits is given in Section~\ref{sec:3}. This permits us to only consider the equivalence classes of $p$-periodic orbits and considerably reduce the number of orbits to analyse in the sequel.  Let $(i_0,j_0)$ be a fixed position in the orbit $\orb{S}$ and $r\in\{0,1,\ldots,p-1\}$ be a fixed residue class modulo $p$. In Section~\ref{sec:4}, necessary and sufficient conditions on the family of Steinhaus triangles $\ST S(i_0,j_0,pk+r)$, for being balanced for all non-negative integers $k$, are obtained. This leads to the proof of Theorem~\ref{thm1} in Section~\ref{sec:5}. Finally, we show in Section~\ref{sec:6} that already known results on balanced triangles modulo $m$ can also be expressed by periodic balanced triangles. 

\section{Periodic orbits}\label{sec:2}

For any $n_1$-tuple $X_1=(a_0,a_1,\ldots,a_{n_1-1})$ and any $n_2$-tuple $X_2=(b_0,b_1,\ldots,b_{n_2-1})$ of elements in $\Zn{2}$, the concatenation $X_1.X_2$ is the $(n_1+n_2)$-tuple $(a_0,a_1,\ldots,a_{n_1-1},b_0,b_1,\ldots,b_{n_2-1})$. For any $n$-tuple $X$, the $kn$-tuple $X^k$ is recursively defined by $X^k=X.X^{k-1}$ for all integers $k\ge 2$, with $X^1=X$. For any $n$-tuple $X=(a_0,a_1,\ldots,a_{n-1})$, the doubly infinite sequence $X^{\infty}=(b_j)_{j\in\Z}$ is defined by $b_{kn+j} = a_{j}$ for all $k\in\Z$ and for all $j\in\{0,1,\ldots,n-1\}$. For any doubly infinite sequence $S=(a_j)_{j\in\Z}$ and any positive integer $n$, we denote by $S[n]$ the initial segment of length $n$ of $S$, that is, the $n$-tuple $S[n]=(a_0,a_1,\ldots,a_{n-1})$.

Let $p$ be a positive integer and let $S=(a_j)_{j\in\Z}$ be a doubly infinite sequence of elements of $\Zn{2}$. The sequence $S$ is said to be {\em periodic of period $p$}, or {\em $p$-periodic}, if $a_{j+p}=a_j$ for all $j\in\Z$. The $p$-periodicity of $S$ is denoted by $S=\left(a_0,a_1,\ldots,a_{p-1}\right)^{\infty}$, where the $p$-tuple $\left(a_0,a_1,\ldots,a_{p-1}\right)$ is a {\em period} of length $p$ of $S$.

First, it is clear that the periodicity of $S$ is preserved under the derivation process.

\begin{prop}\label{prop1}
For any $p$-tuple $\left(a_0,a_1,\ldots,a_{p-1}\right)$, we have
$$
\partial \left(a_0,a_1,\ldots,a_{p-1}\right)^{\infty} = \left(a_{p-1}+a_{0},a_{0}+a_{1},\ldots,a_{p-2}+a_{p-1}\right)^{\infty}
$$
\end{prop}

An infinite sequence $(A_{i})_{i\in\N}$ is said to be {\em pseudo-periodic of period $p$} if there exists $i_0\in\N$ such that $A_{i+p}=A_{i}$ for all $i\ge i_0$.

\begin{prop}
The orbit of a periodic sequence is a pseudo-periodic sequence.
\end{prop}

\begin{proof}
Let $S$ be a $p$-periodic sequence of $\Zn{2}$ and let $\orb{S}=\left(\partial^{i}S\right)_{i\in\N}$ be its associated orbit. By Proposition~\ref{prop1}, we know that, for every non-negative integer $i$, the derived sequence $\partial^{i}S$ is a $p$-periodic sequence. Since the number of $p$-tuples over $\Zn{2}$ and thus the number of $p$-periodic sequences of $\Zn{2}$ is finite, we deduce that there exist $0\le i_1<i_2$ such that $\partial^{i_1}S=\partial^{i_2}S$. This leads to
$$
\partial^{i+(i_2-i_1)}S = \partial^{i-i_1}\partial^{i_2}S = \partial^{i-i_1}\partial^{i_1}S = \partial^{i}S
$$
for all $i\ge i_1$. The sequence $\orb{S}$ is then a pseudo-periodic sequence of period $i_2-i_1$.
\end{proof}

We can retrieve the case where the orbit is pseudo-periodic and not periodic in \citet{Harborth:1972aa,Eliahou:2004aa}. Here, we will study the special case where the orbit is fully periodic.

The orbit $\orb{S}=\left(a_{i,j}\right)_{(i,j)\in\N\times\Z}$ is said to be {\em periodic of period $p$}, or {\em $p$-periodic}, if every row and every column is a $p$-periodic sequence, i.e., if the equalities
$$
a_{i,j+p}=a_{i,j}\quad\text{ and }\quad a_{i+p,j}=a_{i,j}
$$
hold for all $i\in\N$ and all $j\in\Z$. In other words, the orbit $\left(a_{i,j}\right)_{(i,j)\in\N\times\Z}$ is $p$-periodic if the equality
$$
a_{i,j} = a_{\res{i},\res{j}}
$$
holds, for all $(i,j)\in\N\times\Z$, where $\res{x}$ is the remainder in the euclidean division of $x$ by $p$. For example, as depicted in Figure~\ref{fig3}, the orbit $\orb{X^{\infty}}$ associated with the $6$-tuple $X=010100$ is $6$-periodic. Note that a binary triangle appearing in a periodic orbit is simply a periodic binary triangle, as defined above.

\begin{figure}[htbp]
\begin{center}
\begin{tikzpicture}[scale=0.4]\small
\draw[fill=\clr] (6,0) -- (12,0) -- (12,-6) -- (6,-6) -- (6,0);
\draw[fill=\clr] (0,-6) -- (6,-6) -- (6,-12) -- (0,-12) -- (0,-6);
\draw[fill=\clr] (12,-6) -- (18,-6) -- (18,-12) -- (12,-12) -- (12,-6);
\grille{0}{0}{{{0,1,0,1,0,0,0,1,0,1,0,0,0,1,0,1,0,0},
{0,1,1,1,1,0,0,1,1,1,1,0,0,1,1,1,1,0},
{0,1,0,0,0,1,0,1,0,0,0,1,0,1,0,0,0,1},
{1,1,1,0,0,1,1,1,1,0,0,1,1,1,1,0,0,1},
{0,0,0,1,0,1,0,0,0,1,0,1,0,0,0,1,0,1},
{1,0,0,1,1,1,1,0,0,1,1,1,1,0,0,1,1,1},
{0,1,0,1,0,0,0,1,0,1,0,0,0,1,0,1,0,0},
{0,1,1,1,1,0,0,1,1,1,1,0,0,1,1,1,1,0},
{0,1,0,0,0,1,0,1,0,0,0,1,0,1,0,0,0,1},
{1,1,1,0,0,1,1,1,1,0,0,1,1,1,1,0,0,1},
{0,0,0,1,0,1,0,0,0,1,0,1,0,0,0,1,0,1},
{1,0,0,1,1,1,1,0,0,1,1,1,1,0,0,1,1,1}}}
\end{tikzpicture}
\caption{The $6$-periodic orbit $\orb{{010100}^\infty}$}\label{fig3}
\end{center}
\end{figure}

Any square $P_{i_0,j_0}=\left(a_{i_0+i,j_0+j}\right)_{0\le i,j\le p-1}$ of size $p$ is said to be a {\em period} of the $p$-periodic orbit $\orb{S}$. Remark that all the periods of a $p$-periodic orbit have the same multiplicity function, i.e., we have $\m_{P_{i_0,j_0}}=\m_{P_{0,0}}$ for all $(i_0,j_0)\in\N\times\Z$.

The set of $p$-tuples of $\Zn{2}$ that generate $p$-periodic orbits is given in the following theorem, that also appears in \citet{Harborth:1972aa}.

\begin{thm}\label{thm2}
The orbit $\orb{{X}^{\infty}}$ associated with the $p$-tuple $X=(a_0,a_1,\ldots,a_{p-1})$ is $p$-periodic if and only if the vector $v_X=(a_0,a_1,\ldots,a_{p-1})^{t}$ is in the kernel of the matrix $W_p$ which is the Wendt matrix of size $p$ modulo $2$, i.e., the circulant matrix of the binomial coefficients modulo $2$
$$
W_p = \begin{pmatrix}
{p\choose p} & {p\choose{p-1}} & {p\choose{p-2}} & \cdots & {p\choose 1} \\[1.5ex]
{p\choose 1} & {p\choose p} & {p\choose{p-1}} & \cdots & {p\choose 2} \\[1.5ex]
\vdots & \vdots & \vdots & & \vdots \\[1.5ex]
{p\choose{p-1}} & {p\choose{p-2}} & {p\choose{p-3}} & \cdots & {p\choose p} \\
\end{pmatrix}.
$$
\end{thm}

To prove this result, we use the following lemma where it is shown that each term of the orbit $\orb{S}$ can be expressed as a function of the elements of the sequence $S$. The proof of this lemma is a straightforward induction.

\begin{lem}\label{lem1}
Let $\orb{S}=\left(a_{i,j}\right)_{(i,j)\in\N\times\Z}$ be an orbit and let $i_0$ be a non-negative integer. Then,
$$
a_{i_0+i,j} = \sum_{k=0}^{i}{i\choose k}a_{i_0,j-k}
$$
for all $(i,j)\in\N\times\Z$.
\end{lem}

\begin{proof}[of Theorem~\ref{thm2}]
Let $X=(a_0,a_1,\ldots,a_{p-1})$ be a $p$-tuple and let $\orb{{X}^{\infty}}=\left(a_{i,j}\right)_{(i,j)\in\N\times\Z}$ be its associated orbit. Since, by definition, the sequence $X^{\infty}$ is $p$-periodic, we already know from Proposition~\ref{prop1} that the derived sequences $\partial^{i}X^{\infty}$ are $p$-periodic for all non-negative integers $i$. Therefore the equality $a_{i,j+p}=a_{i,j}$ is true for all $(i,j)\in\N\times\Z$. Thus, the orbit $\orb{X^{\infty}}$ is $p$-periodic if and only if $a_{i+p,j}=a_{i,j}$ for all $(i,j)\in\N\times\Z$. The equality $a_{i+p,j}=a_{i,j}$ holds for all $(i,j)\in\N\times\Z$ if and only if $a_{p,j}=a_{0,j}$ for all $j\in\Z$. Moreover, since the sequences $(a_{0,j})_{j\in\Z}$ and $(a_{p,j})_{j\in\Z}$ are $p$-periodic, the orbit $\orb{X^{\infty}}$ is $p$-periodic if and only if $a_{p,j}=a_{0,j}$ for all $j\in\{0,1,\ldots,p-1\}$. From Lemma~\ref{lem1}, we know that $a_{p,j}=\sum_{k=0}^{p}{p\choose k}a_{0,j-k}$ for all $j\in\{0,1,\ldots,p-1\}$. Therefore the orbit $\orb{X^{\infty}}$ is $p$-periodic if and only if
$$
\left\{\begin{array}{r}
{p\choose 1}a_{0,-1} + {p\choose 2}a_{0,-2} + \cdots + {p \choose p}a_{0,-p}  =  0 \\[2ex]
{p\choose 1}a_{0,0} + {p\choose 2}a_{0,-1} + \cdots + {p \choose p}a_{0,-p+1}  =  0 \\
\vdots\quad\quad \\
{p\choose 1}a_{0,p-2} + {p\choose 2}a_{0,p-3} + \cdots + {p \choose p}a_{0,-1}  = 0 \\
\end{array}\right.
\Longleftrightarrow
\left\{\begin{array}{r}
{p\choose 1}a_{p-1} + {p\choose 2}a_{p-2} + \cdots + {p \choose p}a_{0} = 0 \\[2ex]
{p\choose 1}a_{0} + {p\choose 2}a_{p-1} + \cdots + {p \choose p}a_{1} = 0 \\
\vdots\quad\quad \\
{p\choose 1}a_{p-2} + {p\choose 2}a_{p-3} + \cdots + {p \choose p}a_{p-1} = 0 \\
\end{array}\right.
$$
i.e., if and only if the $p$-tuple $X$ is in the kernel of the Wendt matrix $W_p=\left({p\choose {|i-j|}}\right)_{1\le i,j\le p}$ modulo $2$.
\end{proof}

The set of $p$-tuples $X$ that generate $p$-periodic orbits $\orb{X^{\infty}}$ is denoted by $\PO{p}$. It is then a $\Zn{2}$-vector space isomorphic to the kernel of the Wendt matrix $W_p$ of size $p$ modulo $2$. Table~\ref{tab1} gives $\dim\ker(W_p)$ and $|\PO{p}|=2^{\dim\ker(W_p)}$ for the first few values of $p$.

\begin{table}[htbp]
\begin{center}
{\tabulinesep=1mm
\begin{tabu}{|c||c|c|c|c|c|c|c|c|c|c|c|c|}
\hline
 $p$ & $1$ & $2$ & $3$ & $4$ & $5$ & $6$ & $7$ & $8$ & $9$ & $10$ & $11$ & $12$ \\
\hline
$\dim\ker(W_p)$ & $0$ & $0$ & $2$ & $0$ & $0$ & $4$ & $6$ & $0$ & $2$ & $0$ & $0$ & $8$ \\
\hline
\hline
 $p$ & $13$ & $14$ & $15$ & $16$ & $17$ & $18$ & $19$ & $20$ & $21$ & $22$ & $23$ & $24$ \\
\hline
$\dim\ker(W_p)$ &  $0$ & $12$ & $14$ & $0$ & $0$ & $4$ & $0$ & $0$ & $8$ & $0$ & $0$ & $16$ \\
\hline
\end{tabu}
}
\caption{The first few values of $\dim\ker(W_p)$}\label{tab1}
\end{center}
\end{table}

For example, for $p=6$, we have $\dim\ker(W_6)=4$ and $|\PO{6}|=2^{4}=16$. There are then $16$ different $6$-tuples that generate a $6$-periodic orbit. More precisely, the set $\PO{6}$ is given by
$$
\PO{6} \begin{array}[t]{l}
= \displaystyle\left\langle 000101 , 001010 , 010001 , 100010 \right\rangle \\
= \displaystyle\left\{ 000000 , 000101 , 001010 , 001111 , 010001 , 010100 , 011011 , 011110 , \right. \\
\left.\quad\ 100010 ,  100111 , 101000 , 101101 , 110011 , 110110 , 111001 , 111100 \right\}. \\
\end{array}
$$
We recall here that the $6$-tuple $X=010100$ generates a $6$-periodic orbit as depicted in Figure~\ref{fig3}.

\section{Symmetry group of $\PO{p}$}\label{sec:3}

In this section, a symmetry group on the set of $p$-tuples that generate $p$-periodic orbits is defined. First, the notion of translation and the action of the dihedral group $D_3$ on periodic orbits are introduced.

\subsection{Translation}

Let $\orb{X^\infty}=(a_{\res{i},\res{j}})_{(i,j)\in\N\times\Z}$ be the $p$-periodic orbit associated with $X=(a_0,a_1,\ldots,a_{p-1})\in\PO{p}$. The {\em translate of $X$ by the vector $(u,v)\in\Z^2$} is the $p$-tuple $\ts{u}{v}(X) = (a_{\res{-u},\res{j-v}})_{0\le j\le p-1}$. From Lemma~\ref{lem1}, we know that
$$
\ts{u}{v}(X) = \left(\sum_{k=0}^{\res{-u}}{\res{-u}\choose k}a_{\res{j-v-k}}\right)_{0\le j\le p-1}.
$$
From the definition of $\ts{u}{v}(X)$, it is clear that
$$
\orb{{\ts{u}{v}(X)}^\infty} = \left( a_{\res{i-u},\res{j-v}} \right)_{(i,j)\in\N\times\Z}.
$$
Therefore $\ts{u}{v}$ is an automorphism of $\PO{p}$. Moreover, the application
$$
\begin{array}{ccc}
\left(\Z^2,+\right) & \longrightarrow & \left(\mathrm{Aut}(\PO{p}),\circ\right) \\
(u,v) & \longmapsto & \ts{u}{v}
\end{array}
$$
is a group morphism.

For example, the translate of the $6$-tuple $010100\in\PO{6}$ (Figure~\ref{fig3}) by the vector $(2,3)$ is $\ts{2}{3}(010100)=101000$, as we can see in its orbit $\orb{{\ts{2}{3}(010100)}^{\infty}}=\orb{101000^{\infty}}$ depicted in Figure~\ref{fig4}.

\begin{figure}[htbp]
\begin{center}
\begin{tikzpicture}[scale=0.4]\small
\draw[fill=\clr] (3,0) -- (9,0) -- (9,-2) -- (3,-2) -- (3,0);
\draw[fill=\clr] (15,0) -- (18,0) -- (18,-2) -- (15,-2) -- (15,0);
\draw[fill=\clr] (0,-2) -- (3,-2) -- (3,-8) -- (0,-8) -- (0,-2);
\draw[fill=\clr] (9,-2) -- (15,-2) -- (15,-8) -- (9,-8) -- (9,-2);
\draw[fill=\clr] (3,-8) -- (9,-8) -- (9,-12) -- (3,-12) -- (3,-8);
\draw[fill=\clr] (15,-8) -- (18,-8) -- (18,-12) -- (15,-12) -- (15,-8);
\grille{0}{0}{{{1,0,1,0,0,0,1,0,1,0,0,0,1,0,1,0,0,0},
{1,1,1,1,0,0,1,1,1,1,0,0,1,1,1,1,0,0},
{1,0,0,0,1,0,1,0,0,0,1,0,1,0,0,0,1,0},
{1,1,0,0,1,1,1,1,0,0,1,1,1,1,0,0,1,1},
{0,0,1,0,1,0,0,0,1,0,1,0,0,0,1,0,1,0},
{0,0,1,1,1,1,0,0,1,1,1,1,0,0,1,1,1,1},
{1,0,1,0,0,0,1,0,1,0,0,0,1,0,1,0,0,0},
{1,1,1,1,0,0,1,1,1,1,0,0,1,1,1,1,0,0},
{1,0,0,0,1,0,1,0,0,0,1,0,1,0,0,0,1,0},
{1,1,0,0,1,1,1,1,0,0,1,1,1,1,0,0,1,1},
{0,0,1,0,1,0,0,0,1,0,1,0,0,0,1,0,1,0},
{0,0,1,1,1,1,0,0,1,1,1,1,0,0,1,1,1,1}}}
\end{tikzpicture}
\caption{The translate $\ts{2}{3}(010100)=101000$}\label{fig4}
\end{center}
\end{figure}

\subsection{The dihedral group $D_3$}

First, consider the Steinhaus triangles $\ST{S}=(a_{i,j})_{0\le i\le j\le n-1}$ of size $n$. The {\em left and right sides} of $\ST{S}$ are the sequences $\ls(S)=(a_{n-1-i,n-1-i})_{0\le i\le n-1}$ and $\rs(S)=(a_{i,n-1})_{0\le i\le n-1}$, respectively. From Lemma~\ref{lem1}, we know that $\ls(S)$ and $\rs(S)$ can be expressed as functions of the elements of $S=(a_{j})_{0\le j\le n-1}$
$$
\ls(S) = \left(\sum_{k=0}^{n-1-i}{n-1-i\choose k}a_{n-1-i-k}\right)_{0\le i\le n-1}
\quad\text{and}\quad
\rs(S) = \left(\sum_{k=0}^{i}{i\choose k}a_{n-1-k}\right)_{0\le i\le n-1}.
$$
The {reversed sequence} of $S$ is the sequence read from the right to the left, that is $\is(S)=(a_{n-1-j})_{0\le j\le n-1}$.

Due to the symmetries involved in the local rule that generates $\ST{S}$, the Pascal local rule modulo $2$, it is known that the Steinhaus triangles $\ST{\ls(S)}$, $\ST{\rs(S)}$ and $\ST{\is(S)}$ correspond to the rotations of $\mp 120$ degrees around the center of the triangle $\ST{S}$ and the reflection across the vertical line through the center of $\ST{S}$, respectively. More precisely, for all integers $i$ and $j$ such that $1\le i\le j\le n-1$, we have
\begin{eqnarray}\label{eqn1}
a_{i-1,j-1} + a_{i-1,j} = a_{i,j}\ \Longleftrightarrow\ a_{i-1,j} + a_{i,j} = a_{i-1,j-1}\ \Longleftrightarrow\ a_{i-1,j-1} + a_{i,j} = a_{i-1,j}.
\end{eqnarray}
Therefore
$$
\begin{array}{l}
\ST{\ls(S)} = \ST{(a_{n-1-j,n-1-j})_{0\le j\le n-1}} =  (a_{n-1-j,n-1-j+i})_{0\le i\le j\le n-1},\\[2ex]
\ST{\rs(S)} = \ST{(a_{j,n-1})_{0\le j\le n-1}} =  (a_{j-i,n-1-i})_{0\le i\le j\le n-1},\\[2ex]
\ST{\is(S)} = \ST{(a_{0,n-1-j})_{0\le j\le n-1}} = (a_{i,n-1+i-j})_{0\le i\le j\le n-1}.
\end{array}
$$
Since
$$
\rs^3 = \is^2 = (\is\rs)^2 = id_{(\Zn{2})^n},
$$
the subgroup of $\left(\mathrm{Aut}((\Zn{2})^n),\circ\right)$, the group of automorphisms of the vector space of $n$-tuples over $\Zn{2}$, generated by $\rs$ and $\is$ is isomorphic to the dihedral group $D_3$
$$
\left\langle \rs,\is \middle| \rs^3=\is^2=(\is\rs)^2=id_{(\Zn{2})^n}\right\rangle = D_3.
$$

As depicted in Figure~\ref{fig6}, it is easy to see that the multiplicity function of a Steinhaus triangle is invariant under the action of the dihedral group $D_3$. Indeed, for any finite sequence $S$, we have
$$
\m_{\ST{S}} = \m_{\ST{\rs(S)}} = \m_{\ST{\is(S)}}.
$$

\begin{figure}[htbp]
\begin{center}
\begin{tabular}{cccccc}
\begin{tikzpicture}[scale=0.2]
\pgfmathparse{sqrt(3)}\let\x\pgfmathresult

\node at (0,0) {\color{red}{$\mathbf{0}$}};
\node at (2,0) {$1$};
\node at (4,0) {$0$};
\node at (6,0) {\color{green}{$\mathbf{0}$}};

\node at (1,-\x) {$1$};
\node at (3,-\x) {$1$};
\node at (5,-\x) {$0$};

\node at (2,-2*\x) {$0$};
\node at (4,-2*\x) {$1$};

\node at (3,-3*\x) {\color{blue}{$\mathbf{1}$}};

\draw (-1.5,0.5*\x) -- (7.5,0.5*\x) -- (3,-4*\x) -- (-1.5,0.5*\x);

\draw[color=red, line width = 0.1pt] (3,\x) -- (3,-4.5*\x);
\draw[color=red, line width = 0.1pt] (-2.25,0.75*\x) -- (6,-2*\x);
\draw[color=red, line width = 0.1pt] (8.25,0.75*\x) -- (0,-2*\x);

\end{tikzpicture}
&
\begin{tikzpicture}[scale=0.2]
\pgfmathparse{sqrt(3)}\let\x\pgfmathresult

\draw[color=white, line width = 0.1pt] (3,\x) -- (3,-4.5*\x);

\node at (0,0) {\color{green}{$\mathbf{0}$}};
\node at (2,0) {$0$};
\node at (4,0) {$1$};
\node at (6,0) {\color{blue}{$\mathbf{1}$}};

\node at (1,-\x) {$0$};
\node at (3,-\x) {$1$};
\node at (5,-\x) {$0$};

\node at (2,-2*\x) {$1$};
\node at (4,-2*\x) {$1$};

\node at (3,-3*\x) {\color{red}{$\mathbf{0}$}};

\draw (-1.5,0.5*\x) -- (7.5,0.5*\x) -- (3,-4*\x) -- (-1.5,0.5*\x);

\end{tikzpicture}
&
\begin{tikzpicture}[scale=0.2]
\pgfmathparse{sqrt(3)}\let\x\pgfmathresult

\draw[color=white, line width = 0.1pt] (3,\x) -- (3,-4.5*\x);

\node at (0,0) {\color{blue}{$\mathbf{1}$}};
\node at (2,0) {$0$};
\node at (4,0) {$1$};
\node at (6,0) {\color{red}{$\mathbf{0}$}};

\node at (1,-\x) {$1$};
\node at (3,-\x) {$1$};
\node at (5,-\x) {$1$};

\node at (2,-2*\x) {$0$};
\node at (4,-2*\x) {$0$};

\node at (3,-3*\x) {\color{green}{$\mathbf{0}$}};

\draw (-1.5,0.5*\x) -- (7.5,0.5*\x) -- (3,-4*\x) -- (-1.5,0.5*\x);

\end{tikzpicture}
&
\begin{tikzpicture}[scale=0.2]
\pgfmathparse{sqrt(3)}\let\x\pgfmathresult

\draw[color=white, line width = 0.1pt] (3,\x) -- (3,-4.5*\x);

\node at (0,0) {\color{green}{$\mathbf{0}$}};
\node at (2,0) {$0$};
\node at (4,0) {$1$};
\node at (6,0) {\color{red}{$\mathbf{0}$}};

\node at (1,-\x) {$0$};
\node at (3,-\x) {$1$};
\node at (5,-\x) {$1$};

\node at (2,-2*\x) {$1$};
\node at (4,-2*\x) {$0$};

\node at (3,-3*\x) {\color{blue}{$\mathbf{1}$}};

\draw (-1.5,0.5*\x) -- (7.5,0.5*\x) -- (3,-4*\x) -- (-1.5,0.5*\x);

\end{tikzpicture}
&
\begin{tikzpicture}[scale=0.2]
\pgfmathparse{sqrt(3)}\let\x\pgfmathresult

\draw[color=white, line width = 0.1pt] (3,\x) -- (3,-4.5*\x);

\node at (0,0) {\color{red}{$\mathbf{0}$}};
\node at (2,0) {$1$};
\node at (4,0) {$0$};
\node at (6,0) {\color{blue}{$\mathbf{1}$}};

\node at (1,-\x) {$1$};
\node at (3,-\x) {$1$};
\node at (5,-\x) {$1$};

\node at (2,-2*\x) {$0$};
\node at (4,-2*\x) {$0$};

\node at (3,-3*\x) {\color{green}{$\mathbf{0}$}};

\draw (-1.5,0.5*\x) -- (7.5,0.5*\x) -- (3,-4*\x) -- (-1.5,0.5*\x);

\end{tikzpicture}
&
\begin{tikzpicture}[scale=0.2]
\pgfmathparse{sqrt(3)}\let\x\pgfmathresult

\draw[color=white, line width = 0.1pt] (3,\x) -- (3,-4.5*\x);

\node at (0,0) {\color{blue}{$\mathbf{1}$}};
\node at (2,0) {$1$};
\node at (4,0) {$0$};
\node at (6,0) {\color{green}{$\mathbf{0}$}};

\node at (1,-\x) {$0$};
\node at (3,-\x) {$1$};
\node at (5,-\x) {$0$};

\node at (2,-2*\x) {$1$};
\node at (4,-2*\x) {$1$};

\node at (3,-3*\x) {\color{red}{$\mathbf{0}$}};

\draw (-1.5,0.5*\x) -- (7.5,0.5*\x) -- (3,-4*\x) -- (-1.5,0.5*\x);

\end{tikzpicture} \\
$\ST{S}$ & $\ST{\rs(S)}$ & $\ST{\rs^2(S)}$ & $\ST{\is(S)}$ & $\ST{\rs\is(S)}$ & $\ST{\rs^2\is(S)}$ \\
\end{tabular}
\caption{Action of $D_3$ on $\ST{(0100)}$}\label{fig6}
\end{center}
\end{figure}

The study of rotationally symmetric triangles and dihedrally symmetric triangles, that are triangles $\ST{S}$ such that $S=\rs(S)$ and $S=\rs(S)=\is(S)$, respectively, can be found in \citet{Barbe:2000aa,Brunat:2011aa}.

Now, we consider the restrictions of $\rs$ and $\is$ to the vector space $\PO{p}$ of $p$-tuples that generate $p$-periodic orbits. Since we only consider these restrictions, they are also denoted by $\rs$ and $\is$  in the sequel.

\begin{prop}
For all positive integers $p$, we have
$$
\rs\left(\PO{p}\right) = \is\left(\PO{p}\right) = \PO{p}.
$$
\end{prop}

\begin{proof}
Let $X\in\PO{p}$ and $\orb{X^\infty} = (a_{i,j})_{(i,j)\in\N\times\Z}$. Then, by definition, $\rs(X)=(a_{j,p-1})_{0\le j\le p-1}$ and $\is(X)=(a_{0,p-1-j})_{0\le j\le p-1}$. Let $\orb{{\rs(X)}^\infty}=(b_{i,j})_{(i,j)\in\N\times\Z}$ and $\orb{{\is(X)}^\infty}=(c_{i,j})_{(i,j)\in\N\times\Z}$. We will show that $b_{i,j} = a_{\res{j-i},\res{-i-1}}$ and $c_{i,j} = a_{\res{i},\res{i-j-1}}$ for all $i\in\N$ and all $j\in\Z$. We proceed by induction on $i\in\N$. For $i=0$, by definition, we have that
$$
b_{0,j} = b_{0,\res{j}} = a_{\res{j},p-1} = a_{\res{j},\res{-1}}
$$
and
$$
c_{0,j} = c_{0,\res{j}} = a_{0,p-1-\res{j}} = a_{\res{0},\res{-j-1}}
$$
for all $j\in\Z$. Suppose that the formulas are verified for a certain value of $i-1\ge 0$ and for all $j\in\Z$. Then, using \eqref{eqn1},
$$
b_{i,j} = b_{i-1,j-1} + b_{i-1,j} = a_{\res{j-i},\res{-i}} + a_{\res{j-i+1},\res{-i}} = a_{\res{j-i},\res{-i}} + a_{\res{j-i}+1,\res{-i}} = a_{\res{j-i},\res{-i}-1} = a_{\res{j-i},\res{-i-1}}
$$
and
$$
c_{i,j} = c_{i-1,j-1} + c_{i-1,j} \begin{array}[t]{l}
= a_{\res{i-1},\res{i-j-1}} + a_{\res{i-1},\res{i-j-2}} \\
= a_{\res{i-1},\res{i-j-1}} + a_{\res{i-1},\res{i-j-1}-1} = a_{\res{i-1}+1,\res{i-j-1}} = a_{\res{i},\res{i-j-1}} \\
\end{array}
$$
for all $j\in\Z$. This verifies the formulas and we deduce that $b_{i,j}=b_{\res{i},\res{j}}$ and $c_{i,j}=c_{\res{i},\res{j}}$ for all $i\in\N$ and all $j\in\Z$. Then the orbit $\orb{{\rs(X)}^\infty}$ and $\orb{{\is(X)}^\infty}$ are $p$-periodic. Therefore $\rs(X)$ and $\is(X)$ are in $\PO{p}$. This proves that $\rs\left(\PO{p}\right)\subset\PO{p}$ and $\is\left(\PO{p}\right)\subset\PO{p}$ and implies that
$$
\PO{p} = \rs^3\left(\PO{p}\right) \subset\rs^2\left(\PO{p}\right)\subset\rs\left(\PO{p}\right)\subset\PO{p}\quad\text{and}\quad \PO{p}=\is^2\left(\PO{p}\right)\subset\is\left(\PO{p}\right)\subset\PO{p}.
$$
This concludes the proof.
\end{proof}

It follows that $\rs$ and $\is$ are automorphisms of the vector space $\PO{p}$ and the subgroup of $\left(\mathrm{Aut}(\PO{p}),\circ\right)$ generated by $\rs$ and $\is$ is also isomorphic to the dihedral group $D_3$
$$
D_3 = \left\langle \rs,\is \right\rangle = \left\{ id_{\PO{p}} , \rs , \rs^2 , \is , \rs\is , \rs^2\is \right\}.
$$
More precisely, for any $p$-tuple $X$, we have
$$
\begin{array}{ll}
\orb{X^\infty}=(a_{\res{i},\res{j}})_{(i,j)\in\N\times\Z} & \orb{{\is(X)}^\infty}=(a_{\res{i},\res{i-j-1}})_{(i,j)\in\N\times\Z} \\
\orb{{\rs(X)}^\infty}=(a_{\res{j-i},\res{-i-1}})_{(i,j)\in\N\times\Z} & \orb{{\rs\is(X)}^\infty}=(a_{\res{-j-1},\res{-i-1}})_{(i,j)\in\N\times\Z} \\
\orb{{\rs^2(X)}^\infty}=(a_{\res{-j-1},\res{i-j-1}})_{(i,j)\in\N\times\Z} & \orb{{\rs^2\is(X)}^\infty}=(a_{\res{j-i},\res{j}})_{(i,j)\in\N\times\Z} \\
\end{array}
$$
For instance, a representation of $\orb{g(010100)^\infty}$ for all $g\in D_3$ is given in Figure~\ref{fig5}.

\begin{figure}[htbp]
\begin{center}
\begin{tabular}{ccc}

\begin{tikzpicture}[scale=0.4]\small
\draw[fill=\clr] (1,-1) -- (1,-2) -- (2,-2) -- (2,-3) -- (3,-3) -- (3,-4) -- (4,-4) -- (4,-5) -- (5,-5) -- (5,-6) -- (6,-6) -- (6,-7) -- (7,-7) -- (7,-1) -- (1,-1);
\draw[fill=\clrF] (1,-1) -- (7,-1) -- (7,-2) -- (1,-2) -- (1,-1);
\draw[fill=\clrd] (1,-2) -- (2,-2) -- (2,-3) -- (3,-3) -- (3,-4) -- (4,-4) -- (4,-5) -- (5,-5) -- (5,-6) -- (6,-6) -- (6,-7) -- (1,-7) -- (1,-1);
\draw[fill=\clrdF] (1,-2) -- (2,-2) -- (2,-7) -- (1,-7) -- (1,-2);
\grille{0}{0}{{{1,1,0,0,1,1,1,1},
{0,\color{white}{0},1,0,1,0,0,0},
{0,\color{white}{0},1,1,1,1,0,0},
{1,0,1,0,0,0,1,0},
{1,1,1,1,0,0,1,1},
{1,0,0,0,1,0,1,0},
{1,1,0,0,1,1,1,1},
{0,0,1,0,1,0,0,0}}}
\end{tikzpicture}
&
\begin{tikzpicture}[scale=0.4]\small
\draw[fill=\clr] (1,-1) -- (1,-2) -- (2,-2) -- (2,-3) -- (3,-3) -- (3,-4) -- (4,-4) -- (4,-5) -- (5,-5) -- (5,-6) -- (6,-6) -- (6,-7) -- (7,-7) -- (7,-1) -- (1,-1);
\draw[fill=\clrF] (1,-1) -- (2,-1) -- (2,-2) -- (1,-2) -- (1,-1);
\draw[fill=\clrF] (2,-2) -- (3,-2) -- (3,-3) -- (2,-3) -- (2,-2);
\draw[fill=\clrF] (3,-3) -- (4,-3) -- (4,-4) -- (3,-4) -- (3,-3);
\draw[fill=\clrF] (4,-4) -- (5,-4) -- (5,-5) -- (4,-5) -- (4,-4);
\draw[fill=\clrF] (5,-5) -- (6,-5) -- (6,-6) -- (5,-6) -- (5,-5);
\draw[fill=\clrF] (6,-6) -- (7,-6) -- (7,-7) -- (6,-7) -- (6,-6);
\draw[fill=\clrd] (1,-2) -- (2,-2) -- (2,-3) -- (3,-3) -- (3,-4) -- (4,-4) -- (4,-5) -- (5,-5) -- (5,-6) -- (6,-6) -- (6,-7) -- (1,-7) -- (1,-1);
\draw[fill=\clrdF] (1,-6) -- (6,-6) -- (6,-7) -- (1,-7) -- (1,-6);
\grille{0}{0}{{{0,0,0,1,0,1,0,0},
{1,0,0,1,1,1,1,0},
{0,1,0,1,0,0,0,1},
{0,1,1,1,1,0,0,1},
{0,1,0,0,0,1,0,1},
{1,1,1,0,0,1,1,1},
{0,\color{white}{0},0,1,0,1,\color{white}{0},0},
{1,0,0,1,1,1,1,0}}}
\end{tikzpicture}
&
\begin{tikzpicture}[scale=0.4]\small
\draw[fill=\clr] (1,-1) -- (1,-2) -- (2,-2) -- (2,-3) -- (3,-3) -- (3,-4) -- (4,-4) -- (4,-5) -- (5,-5) -- (5,-6) -- (6,-6) -- (6,-7) -- (7,-7) -- (7,-1) -- (1,-1);
\draw[fill=\clrF] (6,-1) -- (7,-1) -- (7,-7) -- (6,-7) -- (6,-1);
\draw[fill=\clrd] (1,-2) -- (2,-2) -- (2,-3) -- (3,-3) -- (3,-4) -- (4,-4) -- (4,-5) -- (5,-5) -- (5,-6) -- (6,-6) -- (6,-7) -- (1,-7) -- (1,-1);
\draw[fill=\clrdF] (1,-2) -- (2,-2) -- (2,-3) -- (1,-3) -- (1,-2);
\draw[fill=\clrdF] (2,-3) -- (3,-3) -- (3,-4) -- (2,-4) -- (2,-3);
\draw[fill=\clrdF] (3,-4) -- (4,-4) -- (4,-5) -- (3,-5) -- (3,-4);
\draw[fill=\clrdF] (4,-5) -- (5,-5) -- (5,-6) -- (4,-6) -- (4,-5);
\draw[fill=\clrdF] (5,-6) -- (6,-6) -- (6,-7) -- (5,-7) -- (5,-6);
\grille{0}{0}{{{0,1,1,1,1,0,0,1},
{0,1,0,0,0,1,\color{white}{0},1},
{1,1,1,0,0,1,1,1},
{0,0,0,1,0,1,0,0},
{1,0,0,1,1,1,1,0},
{0,1,0,1,0,0,0,1},
{0,1,1,1,1,\color{white}{0},0,1},
{0,1,0,0,0,1,0,1}}}
\end{tikzpicture} \\

$\orb{X^\infty}$
&
$\orb{{\rs(X)}^\infty}$
&
$\orb{{\rs^2(X)}^\infty}$ \\ \ \\

\begin{tikzpicture}[scale=0.4]\small
\draw[fill=\clr] (1,-1) -- (1,-2) -- (2,-2) -- (2,-3) -- (3,-3) -- (3,-4) -- (4,-4) -- (4,-5) -- (5,-5) -- (5,-6) -- (6,-6) -- (6,-7) -- (7,-7) -- (7,-1) -- (1,-1);
\draw[fill=\clrF] (1,-1) -- (7,-1) -- (7,-2) -- (1,-2) -- (1,-1);
\draw[fill=\clrd] (1,-2) -- (2,-2) -- (2,-3) -- (3,-3) -- (3,-4) -- (4,-4) -- (4,-5) -- (5,-5) -- (5,-6) -- (6,-6) -- (6,-7) -- (1,-7) -- (1,-1);
\draw[fill=\clrdF] (1,-2) -- (2,-2) -- (2,-3) -- (1,-3) -- (1,-2);
\draw[fill=\clrdF] (2,-3) -- (3,-3) -- (3,-4) -- (2,-4) -- (2,-3);
\draw[fill=\clrdF] (3,-4) -- (4,-4) -- (4,-5) -- (3,-5) -- (3,-4);
\draw[fill=\clrdF] (4,-5) -- (5,-5) -- (5,-6) -- (4,-6) -- (4,-5);
\draw[fill=\clrdF] (5,-6) -- (6,-6) -- (6,-7) -- (5,-7) -- (5,-6);
\grille{0}{0}{{{1,1,1,0,0,1,1,1},
{0,0,0,1,0,1,\color{white}{0},0},
{1,\color{white}{0},0,1,1,1,1,0},
{0,1,0,1,0,0,0,1},
{0,1,1,1,1,0,0,1},
{0,1,0,0,0,1,0,1},
{1,1,1,0,0,1,1,1},
{0,0,0,1,0,1,0,0}}}
\end{tikzpicture}
&
\begin{tikzpicture}[scale=0.4]\small
\draw[fill=\clr] (1,-1) -- (1,-2) -- (2,-2) -- (2,-3) -- (3,-3) -- (3,-4) -- (4,-4) -- (4,-5) -- (5,-5) -- (5,-6) -- (6,-6) -- (6,-7) -- (7,-7) -- (7,-1) -- (1,-1);
\draw[fill=\clrF] (1,-1) -- (2,-1) -- (2,-2) -- (1,-2) -- (1,-1);
\draw[fill=\clrF] (2,-2) -- (3,-2) -- (3,-3) -- (2,-3) -- (2,-2);
\draw[fill=\clrF] (3,-3) -- (4,-3) -- (4,-4) -- (3,-4) -- (3,-3);
\draw[fill=\clrF] (4,-4) -- (5,-4) -- (5,-5) -- (4,-5) -- (4,-4);
\draw[fill=\clrF] (5,-5) -- (6,-5) -- (6,-6) -- (5,-6) -- (5,-5);
\draw[fill=\clrF] (6,-6) -- (7,-6) -- (7,-7) -- (6,-7) -- (6,-6);
\draw[fill=\clrd] (1,-2) -- (2,-2) -- (2,-3) -- (3,-3) -- (3,-4) -- (4,-4) -- (4,-5) -- (5,-5) -- (5,-6) -- (6,-6) -- (6,-7) -- (1,-7) -- (1,-1);
\draw[fill=\clrdF] (1,-2) -- (2,-2) -- (2,-7) -- (1,-7) -- (1,-2);
\grille{0}{0}{{{0,0,1,1,1,1,0,0},
{1,\color{white}{0},1,0,0,0,1,0},
{1,1,1,1,0,0,1,1},
{1,0,0,0,1,0,1,0},
{1,1,0,0,1,1,1,1},
{0,0,1,0,1,0,0,0},
{0,\color{white}{0},1,1,1,1,0,0},
{1,0,1,0,0,0,1,0}}}
\end{tikzpicture}
&
\begin{tikzpicture}[scale=0.4]\small
\draw[fill=\clr] (1,-1) -- (1,-2) -- (2,-2) -- (2,-3) -- (3,-3) -- (3,-4) -- (4,-4) -- (4,-5) -- (5,-5) -- (5,-6) -- (6,-6) -- (6,-7) -- (7,-7) -- (7,-1) -- (1,-1);
\draw[fill=\clrF] (6,-1) -- (7,-1) -- (7,-7) -- (6,-7) -- (6,-1);
\draw[fill=\clrd] (1,-2) -- (2,-2) -- (2,-3) -- (3,-3) -- (3,-4) -- (4,-4) -- (4,-5) -- (5,-5) -- (5,-6) -- (6,-6) -- (6,-7) -- (1,-7) -- (1,-1);
\draw[fill=\clrdF] (1,-6) -- (6,-6) -- (6,-7) -- (1,-7) -- (1,-6);
\grille{0}{0}{{{0,1,0,1,0,0,0,1},
{0,1,1,1,1,0,0,1},
{0,1,0,0,0,1,0,1},
{1,1,1,0,0,1,1,1},
{0,0,0,1,0,1,0,0},
{1,0,0,1,1,1,1,0},
{0,1,0,1,0,\color{white}{0},\color{white}{0},1},
{0,1,1,1,1,0,0,1}}}
\end{tikzpicture} \\

$\orb{{\is(X)}^\infty}$
&
$\orb{{\rs\is(X)}^\infty}$
&
$\orb{{\rs^2\is(X)}^\infty}$ \\

\end{tabular}
\caption{Action of $D_3$ on $\orb{010100^\infty}$}\label{fig5}
\end{center}
\end{figure}

\subsection{The symmetry group of $\PO{p}$}

Let $G$ be the subgroup of $\left(\mathrm{Aut}(\PO{p}),\circ\right)$ generated by $\rs$, $\is$, $\ts{1}{0}$ and $\ts{0}{1}$, that is,
$$
G := \left\langle \rs , \is , \ts{1}{0} , \ts{0}{1} \right\rangle.
$$
As in $D_3$, the equality $\is\rs=\rs^2\is$ holds in $G$. The equalities involving the translations are listed below.

\begin{prop}
For all $(u,v)\in\Z^2$, the equalities $\rs\ts{u}{v} = \ts{v-u}{-u}\rs$ and $\is\ts{u}{v}=\ts{u}{u-v}\is$ hold.
\end{prop}

\begin{proof}
Let $(u,v)\in\Z^2$ and $\orb{X^\infty}=(a_{\res{i},\res{j}})_{(i,j)\in\N\times\Z}$ be a $p$-periodic orbit. Then,
$$
\orb{{\rs\ts{u}{v}(X)}^\infty} = \left( a_{\res{j-i+u-v},\res{-i-1+u}} \right)_{(i,j)\in\N\times\Z} = \orb{{\ts{v-u}{-u}\rs(X)}^\infty}
$$
and
$$
\orb{{\is\ts{u}{v}(X)}^\infty} = \left( a_{\res{i-u},\res{i-j-1+v-u}} \right)_{(i,j)\in\N\times\Z} = \orb{{\ts{u}{u-v}\is(X)}^\infty}.
$$
\end{proof}

From these equalities, it is clear that each element $g\in G$ can be uniquely written as
$$
g = \ts{u}{v}\rs^\alpha\is^\beta
$$
with $u,v\in\{0,1,\ldots,p-1\}$, $\alpha\in\{0,1,2\}$ and $\beta\in\{0,1\}$. Therefore $G$ is a group of order $|G|=6p^2$.

\subsection{Equivalence classes of $\PO{p}$}

Now, we consider the binary relation $\sim_G$ on the set $\PO{p}$ defined by $X_1\sim_{G}X_2$ if and only if there exists $g\in G$ such that $X_2=g(X_1)$. Since $G$ is a subgroup of $\left(\mathrm{Aut}(\PO{p}),\circ\right)$, it is clear that $\sim_G$ is an equivalence relation on $\PO{p}$. Therefore, to search balanced triangles, it is sufficient to examine only one representative of each equivalence classe in the set $\resPO{p}:=\PO{p}/\sim_G$. In the sequel, the equivalence class of the tuple $X$ is denoted by $\res{X}$ and the lexicographically smallest tuple $X$ is used as the representative of each equivalence class $\res{X}$.

For example, for $p=6$, $\PO{6}/\sim_G$ consists of $3$ equivalence classes that contain the $16$ tuples of $\PO{6}$ generating $6$-periodic orbits. More precisely,
$$
\resPO{6}=\PO{6}/\sim_G \begin{array}[t]{l}
= \displaystyle\left\{ \left\{000000\right\} , \left\{000101 , 001010 , 001111 , 010001 , 010100 , 011110 , 100010 , \right.\right. \\
\left.\left.\quad\ 100111 , 101000 , 110011 , 111001 , 111100 \right\} , \left\{011011 , 101101 , 110110\right\}\right\} \\
= \displaystyle\left\{ \res{000000} , \res{000101}, \res{011011} \right\} \\
\end{array}
$$
since
$$
\begin{array}{lll}
001010 = \ts{0}{5}(000101) & 100010 = \ts{0}{1}(000101) & 101101 = \ts{0}{1}(011011)\\
001111 = \ts{1}{3}(000101) & 100111 = \ts{1}{4}(000101) & 110110 = \ts{0}{2}(011011)\\
010001 = \ts{0}{2}(000101) & 101000 = \ts{0}{3}(000101) & \\
010100 = \ts{0}{4}(000101) & 110011 = \ts{1}{5}(000101) & \\
011110 = \ts{1}{2}(000101) & 111001 = \ts{1}{0}(000101) & \\
 & 111100 = \ts{1}{1}(000101) & \\
\end{array}
$$
The $6$-periodic orbits associated with these $3$ equivalence classes are depicted in Figure~\ref{fig7}.

\begin{figure}[htbp]
\begin{center}
\begin{tabular}{ccccc}
\begin{tikzpicture}[scale=0.4]\small
\draw[fill=\clr] (0,0) -- (1,0) -- (1,-1) -- (0,-1) -- (0,0);
\draw[fill=\clr] (0,-7) -- (1,-7) -- (1,-8) -- (0,-8) -- (0,-7);
\draw[fill=\clr] (7,0) -- (8,0) -- (8,-1) -- (7,-1) -- (7,0);
\draw[fill=\clr] (7,-7) -- (8,-7) -- (8,-8) -- (7,-8) -- (7,-7);
\draw[fill=\clr] (1,-1) -- (7,-1) -- (7,-7) -- (1,-7) -- (1,-1);
\grille{0}{0}{{{0,0,0,0,0,0,0,0},
{0,0,0,0,0,0,0,0},
{0,0,0,0,0,0,0,0},
{0,0,0,0,0,0,0,0},
{0,0,0,0,0,0,0,0},
{0,0,0,0,0,0,0,0},
{0,0,0,0,0,0,0,0},
{0,0,0,0,0,0,0,0}}}
\end{tikzpicture}
& & 
\begin{tikzpicture}[scale=0.4]\small
\draw[fill=\clr] (0,0) -- (1,0) -- (1,-1) -- (0,-1) -- (0,0);
\draw[fill=\clr] (0,-7) -- (1,-7) -- (1,-8) -- (0,-8) -- (0,-7);
\draw[fill=\clr] (7,0) -- (8,0) -- (8,-1) -- (7,-1) -- (7,0);
\draw[fill=\clr] (7,-7) -- (8,-7) -- (8,-8) -- (7,-8) -- (7,-7);
\draw[fill=\clr] (1,-1) -- (7,-1) -- (7,-7) -- (1,-7) -- (1,-1);
\grille{0}{0}{{{1,1,1,1,0,0,1,1},
{1,0,0,0,1,0,1,0},
{1,1,0,0,1,1,1,1},
{0,0,1,0,1,0,0,0},
{0,0,1,1,1,1,0,0},
{1,0,1,0,0,0,1,0},
{1,1,1,1,0,0,1,1},
{1,0,0,0,1,0,1,0}}}
\end{tikzpicture}
& & 
\begin{tikzpicture}[scale=0.4]\small
\draw[fill=\clr] (0,0) -- (1,0) -- (1,-1) -- (0,-1) -- (0,0);
\draw[fill=\clr] (0,-7) -- (1,-7) -- (1,-8) -- (0,-8) -- (0,-7);
\draw[fill=\clr] (7,0) -- (8,0) -- (8,-1) -- (7,-1) -- (7,0);
\draw[fill=\clr] (7,-7) -- (8,-7) -- (8,-8) -- (7,-8) -- (7,-7);
\draw[fill=\clr] (1,-1) -- (7,-1) -- (7,-7) -- (1,-7) -- (1,-1);
\grille{0}{0}{{{1,1,0,1,1,0,1,1},
{1,0,1,1,0,1,1,0},
{0,1,1,0,1,1,0,1},
{1,1,0,1,1,0,1,1},
{1,0,1,1,0,1,1,0},
{0,1,1,0,1,1,0,1},
{1,1,0,1,1,0,1,1},
{1,0,1,1,0,1,1,0}}}
\end{tikzpicture}
\end{tabular}
\end{center}
\caption{The set $\resPO{6}=\left\{ \res{000000} , \res{000101} , \res{011011} \right\}$}\label{fig7}
\end{figure}

Table~\ref{tab2} gives $\left| \resPO{p} \right|$ for the first few values of $p$.

\begin{table}[htbp]
\begin{center}
{\tabulinesep=1.2mm
\begin{tabu}{|c||c|c|c|c|c|c|c|c|c|c|c|c|}
\hline
 $p$ & $1$ & $2$ & $3$ & $4$ & $5$ & $6$ & $7$ & $8$ & $9$ & $10$ & $11$ & $12$ \\
\hline
$\left| \PO{p} \right|$ & $1$ & $1$ & $2^2$ & $1$ & $1$ & $2^4$ & $2^6$ & $1$ & $2^2$ & $1$ & $1$ & $2^8$ \\
\hline
$\left| \resPO{p} \right|$ & $1$ & $1$ & $2$ & $1$ & $1$ & $3$ & $3$ & $1$ & $2$ & $1$ & $1$ & $7$ \\
\hline
\hline
 $p$ & $13$ & $14$ & $15$ & $16$ & $17$ & $18$ & $19$ & $20$ & $21$ & $22$ & $23$ & $24$ \\
\hline
$\left| \PO{p} \right|$ & $1$ & $2^{12}$ & $2^{14}$ & $1$ & $1$ & $2^{4}$ & $1$ & $1$ & $2^{8}$ & $1$ & $1$ & $2^{16}$ \\
\hline
$\left| \resPO{p} \right|$ & $1$ & $13$ & $30$ & $1$ & $1$ & $3$ & $1$ & $1$ & $6$ & $1$ & $1$ & $92$ \\
\hline
\end{tabu}}
\caption{The first few values of $\left| \resPO{p} \right|$}\label{tab2}
\end{center}
\end{table}

\section{Family of periodic balanced Steinhaus triangles with the same principal vertex}\label{sec:4}

In this section, we determine necessary and sufficient conditions for obtaining, in a $p$-periodic orbit, an infinite family of balanced Steinhaus triangles with the same principal vertex.

\begin{prop}\label{prop2}
Let $S=X^\infty$ with $X\in\PO{p}$, $(i_0,j_0)\in\N\times\Z$ and $r\in\{0,1,\ldots,p-1\}$. The Steinhaus triangles
$$
T_k := \ST{S}(i_0,j_0,kp+r)
$$
are balanced for all non-negative integers $k$ if and only if the triangle $T_0$, the multiset difference $T_1\setminus T_0$ and the period $P$ are balanced, with $p$ divisible by $4$.
\end{prop}

\begin{proof}
Suppose that $(i_0,j_0)\in\N\times\Z$ and $r\in\{0,1,\ldots,p-1\}$ are fixed. Let $k\in\N$. Then, from the periodicity of $\orb{S}$, we know that the Steinhaus triangle $T_k=\ST{S}(i_0,j_0,kp+r)$ can be decomposed into elementary blocks $T_0$, $T_1\setminus T_0$ and $P$, as represented in Figure~\ref{fig8} for $k=5$.

\begin{figure}[htbp]
\tiny
\begin{center}
\begin{tikzpicture}[scale=0.125]
\foreach \x in {0,1,...,44} {
	\draw (\x,-\x) -- (\x,-\x-1) -- (\x+1,-\x-1);
}
\draw (0,0) -- (45,0) -- (45,-45);
\draw (5,0) -- (5,-5);
\draw (13,0) -- (13,-13);
\draw (21,0) -- (21,-21);
\draw (29,0) -- (29,-29);
\draw (37,0) -- (37,-37);
\draw (13,-8) -- (45,-8);
\draw (21,-16) -- (45,-16);
\draw (29,-24) -- (45,-24);
\draw (37,-32) -- (45,-32);
\draw (3.5,-1.5) node {$T_0$};
\draw (9,-4) node {$T_1\setminus T_0$};
\draw (17,-4) node {$P$};
\draw (25,-4) node {$P$};
\draw (33,-4) node {$P$};
\draw (41,-4) node {$P$};
\draw (17,-12) node {$T_1\setminus T_0$};
\draw (25,-12) node {$P$};
\draw (33,-12) node {$P$};
\draw (41,-12) node {$P$};
\draw (25,-20) node {$T_1\setminus T_0$};
\draw (33,-20) node {$P$};
\draw (41,-20) node {$P$};
\draw (33,-28) node {$T_1\setminus T_0$};
\draw (41,-28) node {$P$};
\draw (41,-36) node {$T_1\setminus T_0$};
\end{tikzpicture}
\end{center}
\caption{Decomposition of $T_5$}\label{fig8}
\end{figure}
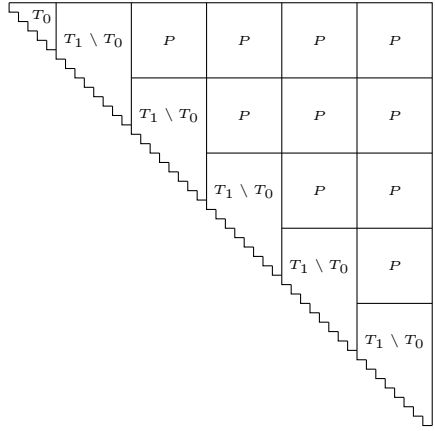

More precisely, the triangle $T_k$ is constituted by one block $T_0$, $k$ blocks $T_1\setminus T_0$ and ${k\choose 2}$ blocks $P$. It follows that the multiplicity function $\m_{T_k}$ of the triangle $T_k$ must verify that
$$
\m_{T_k} = \m_{T_0} + k\m_{T_1\setminus T_0} + {k\choose 2}\m_P.
$$
First, suppose that $T_0$, $T_1\setminus T_0$ and $P$ are balanced with $p$ divisible by $4$. Since $p$ is divisible by $4$, it is clear that the cardinalities $|T_1\setminus T_0|=pr+{p+1\choose 2}$ and $|P|=p^2$ are even. Therefore, since the multiset difference $T_1\setminus T_0$ and the period $P$ are balanced, the multiplicity functions  $\m_{T_1\setminus T_0}$ and $\m_P$ are constant. It follows that $\delta_{T_K}=\delta_{T_0}\in\{0,1\}$ and thus the triangles $T_k$ are balanced for all non-negative integers $k$.

Conversely, suppose that the triangles $T_k$ are balanced for all non-negative integers $k$. Thus, the value of $(\m_{T_k}(0)-\m_{T_k}(1))-(\m_{T_0}(0)-\m_{T_0}(1))$ is in the set $\{-2,-1,0,1,2\}$ for all $k\in\N$. Therefore,
$$
\lim\limits_{k\to\infty}\frac{(\m_{T_k}(0)-\m_{T_k}(1))-(\m_{T_0}(0)-\m_{T_0}(1))}{k} = 0.
$$
It follows that
$$
\lim\limits_{k\to\infty} (\m_{T_1\setminus T_0}(0)-\m_{T_1\setminus T_0}(1)) + \frac{k-1}{2}(\m_{P}(0)-\m_{P}(1)) = 0.
$$
Then, we deduce that $\m_{T_1\setminus T_0}(0)-\m_{T_1\setminus T_0}(1)=\m_{P}(0)-\m_{P}(1)=0$. Therefore the multiset difference $T_1\setminus T_0$ and the period $P$ are balanced and with even cardinalities. Finally, since $|P|=p^2$ and $|T_1\setminus T_0| = pr+{p+1\choose 2}$, we conclude that $p$ must be divisible by $4$ in this case.
\end{proof}

This is the reason why, in the sequel of this paper, we only consider $p$-periodic orbits with a balanced period and where $p$ is divisible by $4$.

Note that the period of the orbit generated from every element of a same equivalence class of $\resPO{p}$ has the same multiplicity function. Let us denote by $\BresPO{p}$ the set of all the equivalence classes of $\resPO{p}$ having a balanced period. Table~\ref{tab3} gives $\left| \BresPO{p} \right|$ for the first few values of $p$ divisible by $4$.

\begin{table}[htbp]
\begin{center}
{\tabulinesep=1.2mm
\begin{tabu}{|c||c|c|c|c|c|c|}
\hline
 $p$ & $4$ & $8$ & $12$ & $16$ & $20$ & $24$ \\
\hline
$\left| \PO{p} \right|$ & $1$ & $1$ & $256$ & $1$ & $1$ & $65536$ \\
\hline
$\left| \resPO{p} \right|$ & $1$ & $1$ & $7$ & $1$ & $1$ & $92$ \\
\hline
$\left| \BresPO{p} \right|$ & $0$ & $0$ & $2$ & $0$ & $0$ & $17$ \\
\hline
\end{tabu}}
\caption{The first few values of $\left| \BresPO{p} \right|$}\label{tab3}
\end{center}
\end{table}

More precisely, we obtain
$$
\BresPO{12} = \left\{\res{Y_1},\res{Y_2}\right\} = \left\{ \res{000001110111} , \res{000101000101} \right\}
$$
and $\BresPO{24}=\left\{\res{X_1},\res{X_2},\ldots,\res{X_{17}}\right\}$, where the $17$ representatives $X_i$ are given in Table~\ref{tab7}. Note that $X_{16}=Y_{1}^{2}$ and $X_{17}=Y_{2}^{2}$. Therefore the orbits $\orb{X_{16}^\infty}$ and $\orb{X_{17}^\infty}$ correspond to $\orb{Y_{1}^\infty}$ and $\orb{Y_{2}^\infty}$, respectively. A representation of the orbits generated from the elements of $\BresPO{12}$ and $\BresPO{24}$ can be found in Appendix~\ref{app1}.

\begin{table}[htbp]
\begin{center}
\begin{tabular}{cc}
\begin{tabular}{|c||c|}
\hline
$i$ & $X_i$ \\
\hline
$1$ & $000000000110101101101011$ \\
\hline
$2$ & $000000001001011110010111$ \\
\hline
$3$ & $000000010010011000100111$ \\
\hline
$4$ & $000000010110111001101111$ \\
\hline
$5$ & $000000010111110001111101$ \\
\hline
$6$ & $000000100001111100011101$ \\
\hline
$7$ & $000000100010100100101011$ \\
\hline
$8$ & $000000100011101100111001$ \\
\hline
$9$ & $000000101000111110001101$ \\
\hline
\end{tabular}
&
\begin{tabular}{|c||c|}
\hline
$i$ & $X_i$ \\
\hline
$10$ & $000000101001110110011111$ \\
\hline
$11$ & $000001000010010100100001$ \\
\hline
$12$ & $000001000101101101011111$ \\
\hline
$13$ & $000001001000001110000111$ \\
\hline
$14$ & $000001001111110111111001$ \\
\hline
$15$ & $000001100001100100011111$ \\
\hline
$16$ & $000001110111000001110111$ \\
\hline
$17$ & $000101000101000101000101$ \\
\hline
 & \\
\hline 
\end{tabular} \\
\end{tabular}
\caption{The representatives $X_i$ of $\BresPO{24}=\left\{\res{X_1},\res{X_2},\ldots,\res{X_{17}}\right\}$}\label{tab7}
\end{center}
\end{table}

\section{Periodic balanced triangles}\label{sec:5}

In this section we will prove Theorem~\ref{thm1}, the main result of this paper.

Let $X$ be a $p$-tuple of $\Zn{2}$, with $p$ divisible by $4$, such that $\res{X}$ is in $\BresPO{p}$ and let $S:=X^\infty$. Now, for each remainder $r\in\{0,1,\ldots,p-1\}$ and for each position $(i_0,j_0)\in{\{0,1,\ldots,p-1\}}^2$, we test if the blocks $\ST{S}(i_0,j_0,r)$ and $\ST{S}(i_0,j_0,p+r)\setminus\ST{S}(i_0,j_0,r)$ are balanced. If this is the case, we know from Proposition~\ref{prop2} that the Steinhaus triangles $\ST{S}(i_0,j_0,kp+r)$ are balanced for all non-negative integers $k$.

Let $\remaind{\res{X}}$ denote the set of remainders $r\in\{0,1,\ldots,p-1\}$ for which there exists a position $(i_0,j_0)\in{\{0,1,\ldots,p-1\}}^2$ such that the Steinhaus triangles $\ST{S}(i_0,j_0,kp+r)$ are balanced for all non-negative integers $k$.

From Table~\ref{tab3}, the first values of $p$, divisible by $4$, for which $\BresPO{p}\neq\emptyset$ are $12$ and $24$. For $p=12$, we find that $\remaind{\res{Y_i}}=\emptyset$ for each of the two equivalence classes $\BresPO{12} = \left\{\res{Y_1},\res{Y_2}\right\}$. For $p=24$, we find that $\remaind{\res{X_i}}\neq\emptyset$ for the first $15$ of the $17$ equivalence classes of $\BresPO{24}=\left\{\res{X_1},\res{X_2},\ldots,\res{X_{17}}\right\}$. Note that the two equivalence classes $\res{X}$ of $\BresPO{24}$ such that $\remaind{\res{X}}=\emptyset$ are exactly of the form $X=Y^2$ with $\res{Y}\in\BresPO{12}$ (as already seen $X_{16}=Y_{1}^{2}$ and $X_{17}=Y_{2}^{2}$). More precisely, Table~\ref{tab4} gives the exact number of remainders constituting $\remaind{\res{X}}$ for each $\res{X}\in\BresPO{24}$.

\begin{table}[htbp]
\begin{center}
{\tabulinesep=1.2mm
\begin{tabu}{|c||c||l|}
\hline
$i$ & $X_i$ & $|\remaind{\res{X_i}}|$ \\
\hline
$1$ & $000000000110101101101011$ & $18$ \\
\hline
$2$ & $000000001001011110010111$ & $16$ \\
\hline
$3$ & $000000010010011000100111$ & $23$ \\
\hline
${\bf 4}$ & ${\bf 000000010110111001101111}$ & ${\bf 24}$ \\
\hline
$5$ & $000000010111110001111101$ & $17$ \\
\hline
${\bf 6}$ & ${\bf 000000100001111100011101}$ & ${\bf 24}$ \\
\hline
${\bf 7}$ & ${\bf 000000100010100100101011}$ & ${\bf 24}$ \\
\hline
${\bf 8}$ & ${\bf 000000100011101100111001}$ & ${\bf 24}$ \\
\hline
${\bf 9}$ & ${\bf 000000101000111110001101}$ & ${\bf 24}$ \\
\hline
$10$ & $000000101001110110011111$ & $23$ \\
\hline
${\bf 11}$ & ${\bf 000001000010010100100001}$ & ${\bf 24}$ \\
\hline
$12$ & $000001000101101101011111$ & $23$ \\
\hline
$13$ & $000001001000001110000111$ & $20$ \\
\hline
$14$ & $000001001111110111111001$ & $20$ \\
\hline
$15$ & $000001100001100100011111$ & $23$ \\
\hline
$16$ & $000001110111000001110111$ & $0$ \\
\hline
$17$ & $000101000101000101000101$ & $0$ \\
\hline
\end{tabu}}
\caption{$|\remaind{\res{X}}|$ for all $\res{X}\in\BresPO{24}$}\label{tab4}
\end{center}
\end{table}

For six equivalence classes $\res{X}$ of $\BresPO{24}$, we find that $|\remaind{\res{X}}|=24$ and thus, from these $24$-tuples, we obtain the proof of Theorem~\ref{thm1} for Steinhaus triangles, i.e., there exist periodic orbits containing balanced Steinhaus triangles of size $n$ for all $n\ge 1$.

For instance, in the orbit $\orb{X_9^\infty}$ associated with the $24$-tuple $X_9=000000101000111110001101$, the existence of balanced Steinhaus triangles for all the possible sizes can be obtained from at least $4$ positions. Table~\ref{tab5} gives positions $(i_0,j_0)$ in the orbit $\orb{X_9^\infty}$ for which the Steinhaus triangles $\ST{X_9^\infty}(i_0,j_0,24k+r)$ are balanced for all non-negative integers~$k$ and the corresponding $24$-tuples $Z$ such that $\ST{Z^\infty[24k+r]}=\ST{X_9^\infty}(i_0,j_0,24k+r)$.

\begin{table}[htbp]
\begin{center}
\begin{tabular}{|c||c||c|}
\hline
$r$ & $(i_0,j_0)$ & $Z$ \\
\hline
$0 , 4 , 7 , 8 , 12 , 13 , 15 , 16 , 21, 22 , 23$ & $(1,11)$ & $010000100101110000011110$ \\
\hline
$1 , 2 , 3 , 5 , 10 , 17 , 18 , 19 , 20 , 21$ & $(1,6)$ & $111100100001001011100000$ \\
\hline
$0 , 1 , 6 , 9 , 14 , 22 , 23$ & $(6,9)$ & $000111010101000101001100$ \\
\hline
$11$ & $(3,3)$ & $000110011101001011001011$ \\
\hline
\end{tabular}
\caption{Balanced Steinhaus triangles $\ST{X_9^\infty}(i_0,j_0,24k+r)=\ST{Z^\infty[24k+r]}$}\label{tab5}
\end{center}
\end{table}

The family of balanced Steinhaus triangles $\ST{X_9^\infty}(6,9,24k+6)$, appearing in the orbit $\orb{X_9^\infty}$ for $X_9=000000101000111110001101$, is depicted in Figure~\ref{fig12}, where empty and full squares correspond to $0$ and $1$ respectively. Indeed, we can verify that the blocks $T_0:=\ST{X_9^\infty}(6,9,6)$, $T_1\setminus T_0:=\ST{X_9^\infty}(6,9,30)\setminus\ST{X_9^\infty}(6,9,6)$ and the period $P$ are balanced, since their multiplicity functions, given in Table~\ref{tab6}, are constant or almost constant.

\begin{table}[htbp]
\begin{center}
\begin{tabular}{|c|c|c|c|}
\hline
$x$ & $\m_{T_0}(x)$ & $\m_{T_1\setminus T_0}(x)$ & $\m_P(x)$ \\
\hline
$0$ & $11$ &$222$ & $288$ \\
\hline
$1$ & $10$ & $222$ & $288$ \\
\hline 
\end{tabular}
\caption{The multiplicity functions of $T_0$, $T_1\setminus T_0$ and $P$}\label{tab6}
\end{center}
\end{table}

\begin{figure}[htbp]
\begin{center}
\includegraphics[width=\textwidth]{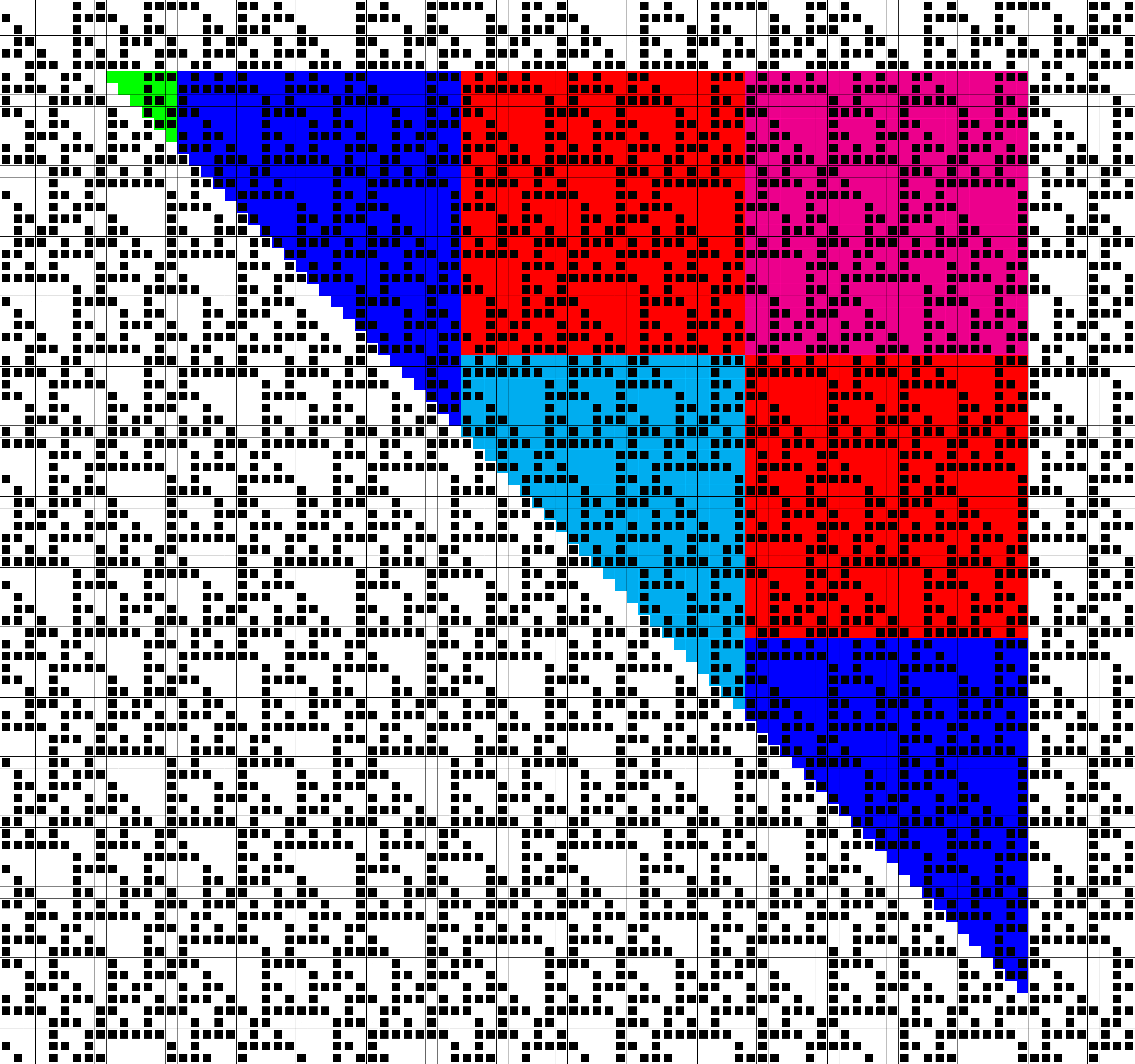}
\caption{The balanced Steinhaus triangles $\ST{X_9^\infty}(6,9,24k+6)$}\label{fig12}
\end{center}
\end{figure}

The following proposition concludes the proof of Theorem~\ref{thm1} by showing that in an orbit $\orb{X^\infty}$ generated from a $p$-tuple $X$ such that $\res{X}\in\BresPO{p}$, the existence of balanced Steinhaus triangles implies that of balanced generalized Pascal triangles.

\begin{prop}\label{prop3}
Let $S=X^\infty$ with $\res{X}\in\BresPO{p}$, $(i_0,j_0)\in\N\times\Z$, $r\in\{0,1,\ldots,p-1\}$ and $p$ divisible by $4$. Then, the Steinhaus triangles $\ST{S}(i_0,j_0,kp+r)$ are balanced for all non-negative integers $k$ if and only if the generalized Pascal triangles $\PT{S}(i_0+r+1,j_0+r,kp+(p-1-r))$ are balanced for all non-negative integers $k$.
\end{prop}

\begin{proof}
As depicted in Figure~\ref{fig9}, we consider in the orbit $\orb{S}=(a_{i,j})_{(i,j)\in\N\times\Z}$ the elementary blocks
$$
\begin{array}{l}
U_0 = \ST{S}(i_0,j_0,r), \\[1.25ex]
U_1 = \ST{S}(i_0,j_0,p+r)\setminus(\ST{S}(i_0,j_0,r)\cup\ST{S}(i_0+p,j_0+p,r)), \\[1.25ex]
V_1 = \PT{S}(i_0',j_0',r'), \\[1.25ex]
V_0 = \PT{S}(i_0',j_0',p+r')\setminus(\PT{S}(i_0',j_0',r')\cup\PT{S}(i_0'+p,j_0'+p,r')), \\
\end{array}
$$
where $i_0'=i_0+r+1$, $j_0'=j_0+r$ and $r'=p-1-r$, and the period
$$
P = \left\{ a_{p+r+i,j}\ \middle|\ i,j\in\{0,1,\ldots,p-1\} \right\}.
$$
Since $\orb{S}$ is $p$-periodic with $P$ balanced and $p$ divisible by $4$, we already know from Proposition~\ref{prop2} that the Steinhaus triangles $\ST{S}(i_0,j_0,kp+r)$ are balanced for all non-negative integers $k$ if and only if $U_0$ and $U_1\cup U_0$ are balanced.

\begin{figure}[htbp]
\begin{center}
\footnotesize
\begin{tikzpicture}[scale=0.125]
\foreach \x in {0,1,...,34} {
	\draw (\x,-\x) -- (\x,-\x-1) -- (\x+1,-\x-1);
}
\draw (0,0) -- (29,0);
\draw (5,-12) -- (29,-12);
\draw (5,-24) -- (29,-24);
\draw (5,-36) -- (35,-36);
\draw (5,0) -- (5,-36);
\draw (17,0) -- (17,-36);
\draw (29,0) -- (29,-36);
\draw (35,-35) -- (35,-36);

\draw[fill=\clr] (0,0) -- (0,-1) -- (1,-1) -- (1,-2) -- (2,-2) -- (2,-3) -- (3,-3) -- (3,-4) -- (4,-4) -- (4,-5) -- (5,-5) -- (5,0) -- (0,0);
\draw[fill=\clrd] (5,0) -- (5,-6) -- (6,-6) -- (6,-7) -- (7,-7) -- (7,-8) -- (8,-8) -- (8,-9) -- (9,-9) -- (9,-10) -- (10,-10) -- (10,-11) -- (11,-11) -- (11,-12) -- (17,-12) -- (17,0) -- (5,0);
\draw[fill=\clrdF] (5,-6) -- (6,-6) -- (6,-7) -- (7,-7) -- (7,-8) -- (8,-8) -- (8,-9) -- (9,-9) -- (9,-10) -- (10,-10) -- (10,-11) -- (11,-11) -- (11,-12) -- (5,-12) -- (5,-6);
\draw[fill=\clrt] (17,0) -- (29,0) -- (29,-12) -- (17,-12) -- (17,0);
\draw[fill=\clrF] (5,-12) -- (12,-12) -- (12,-13) -- (13,-13) -- (13,-14) -- (14,-14) -- (14,-15) -- (15,-15) -- (15,-16) -- (16,-16) -- (16,-17) -- (17,-17) -- (17,-24) -- (5,-24) -- (5,-12);
\draw[fill=\clr] (12,-12) -- (12,-13) -- (13,-13) -- (13,-14) -- (14,-14) -- (14,-15) -- (15,-15) -- (15,-16) -- (16,-16) -- (16,-17) -- (17,-17) -- (17,-12) -- (12,-12);
\draw[fill=\clrd] (17,-12) -- (17,-18) -- (18,-18) -- (18,-19) -- (19,-19) -- (19,-20) -- (20,-20) -- (20,-21) -- (21,-21) -- (21,-22) -- (22,-22) -- (22,-23) -- (23,-23) -- (23,-24) -- (29,-24) -- (29,-12) -- (17,-12);
\draw[fill=\clrdF] (17,-18) -- (18,-18) -- (18,-19) -- (19,-19) -- (19,-20) -- (20,-20) -- (20,-21) -- (21,-21) -- (21,-22) -- (22,-22) -- (22,-23) -- (23,-23) -- (23,-24) -- (17,-24) -- (17,-18);
\draw[fill=\clrt] (5,-24) -- (17,-24) -- (17,-36) -- (5,-36) -- (5,-24);
\draw[fill=\clrF] (17,-24) -- (24,-24) -- (24,-25) -- (25,-25) -- (25,-26) -- (26,-26) -- (26,-27) -- (27,-27) -- (27,-28) -- (28,-28) -- (28,-29) -- (29,-29) -- (29,-36) -- (17,-36) -- (17,-24);
\draw[fill=\clr] (24,-24) -- (24,-25) -- (25,-25) -- (25,-26) -- (26,-26) -- (26,-27) -- (27,-27) -- (27,-28) -- (28,-28) -- (28,-29) -- (29,-29) -- (29,-24) -- (24,-24);
\draw[fill=\clrdF] (29,-30) -- (30,-30) -- (30,-31) -- (31,-31) -- (31,-32) -- (32,-32) -- (32,-33) -- (33,-33) -- (33,-34) -- (34,-34) -- (34,-35) -- (35,-35) -- (35,-36) -- (29,-36) -- (29,-30);
\draw (3.5,-1.5) node {$U_0$};
\draw (15.5,-13.5) node {$U_0$};
\draw (27.5,-25.5) node {$U_0$};
\draw (11,-6) node {$U_1$};
\draw (23,-18) node {$U_1$};
\draw (23,-6) node {$P$};
\draw (6.5,-10.5) node {$V_1$};
\draw (18.5,-22.5) node {$V_1$};
\draw (30.5,-34.5) node {$V_1$};
\draw (11,-18) node {$V_0$};
\draw (23,-30) node {$V_0$};
\draw (11,-30) node {$P$};
\end{tikzpicture}
\caption{Decomposition of $\orb{S}$ into elementary blocks $U_0$, $U_1$, $V_0$ and $V_1$}\label{fig9}
\end{center}
\end{figure}
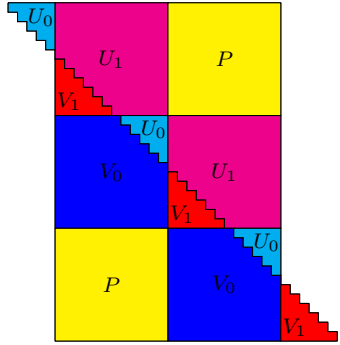

Similarly, the generalized Pascal triangles $\PT{S}(i_0+r+1,j_0+r,kp+p-1-r)$ are balanced for all non-negative integers $k$ if and only if $V_1$ and $V_0\cup V_1$ are balanced.

First, since
$$
U_0\cup V_0 = U_1\cup V_1 = P
$$
is balanced with an even cardinality, then we have $\m_{U_0\cup V_0}(x)=\m_{U_1\cup V_1}(x)=\frac{p^2}{2}$ for all $x\in\Zn{2}$. 
It follows that
$$
\m_{U_0\cup U_1}(x) = \m_{U_0}(x) + \m_{U_1}(x) = (p^2/2-\m_{V_0}(x)) + (p^2/2-\m_{V_1}(x)) = p^2-\m_{V_0\cup V_1}(x)
$$
for all $x\in\Zn{2}$. Therefore $\delta\m_{U_0\cup U_1}=\delta\m_{V_0\cup V_1}$. Moreover,
\begin{eqnarray}\label{eqn2}
\m_{U_0}(x) - \m_{V_1}(x) = (\m_{U_0}(x)+\m_{U_1}(x)) - (\m_{U_1}(x)+\m_{V_1}(x)) = \m_{U_0\cup U_1}(x) - p^2/2
\end{eqnarray}
for all $x\in\Zn{2}$. Since $p$ is divisible by $4$, we already know from Proposition~\ref{prop2} that $U_0\cup U_1$ has an even number of elements. Therefore, if $U_0\cup U_1$ is balanced, then $\delta_{U_0\cup U_1}=0$ and we deduce from \eqref{eqn2} that
$$
\m_{U_0}(1)-\m_{U_0}(0) = \m_{V_1}(1)-\m_{V_1}(0),
$$
that is $\delta_{U_0}=\delta_{V_1}$. Similarly, we can see that if $V_0\cup V_1$ is balanced, then $\delta_{U_0}=\delta_{V_1}$. Therefore the blocks $U_0$ and $U_1\cup U_0$ are balanced if and only if the blocks $V_1$ and $V_0\cup V_1$ are balanced. This concludes the proof.
\end{proof}

Using Proposition~\ref{prop3} and the families of balanced Steinhaus triangles appearing in the orbit $\orb{X_9^\infty}$ associated with the $24$-tuple $X_9=000000101000111110001101$ given in Table~\ref{tab5}, we obtain the existence of balanced generalized Pascal triangles for all the possible sizes. For all non-negative integers $k$, we know that the generalized Pascal triangle $\PT{X_9^\infty}(i_0,j_0,24k+r)$ is balanced for the values of Table~\ref{tab8} since the Steinhaus triangle $\ST{X_9^\infty}(i_0+r-p,j_0+r+1-p,24k+(23-r))$ is. The corresponding $24$-tuples $Z_l$ and $Z_r$ such that $\PT{X_9^\infty}(i_0,j_0,24k+r)=\PT{(Z_l^\infty[24k+r],Z_r^\infty[24k+r])}$ are also given in Table~\ref{tab8}.

\begin{table}[htbp]
\begin{center}
\rotatebox{90}{
\begin{tabular}{|c|c||c|c||c|c|}
\hline
$r$ & $(i_0,j_0)$ & $p-1-r$ & $(i_0+r-p,j_0+r+1-p)$ & $Z_l$ & $Z_r$ \\
\hline
$0$ & $(25,34)$ & $23$ & $(1,11)$ & $011110000101011000101110$ & $001100100000100000111010$ \\
\hline
$1$ & $(24,33)$ & $22$ & $(1,11)$ & $010001000111110100111001$ & $000110010000010000011101$ \\
\hline
$2$ & $(23,32)$ & $21$ & $(1,11)$ & $111001100100001110100101$ & $100011001000001000001110$ \\
\hline
$3$ & $(22,26)$ & $20$ & $(1,6)$ & $110001011100111100001010$ & $110011110010111011100001$ \\
\hline
$4$ & $(21,25)$ & $19$ & $(1,6)$ & $101001110010100010001111$ & $111001111001011101110000$ \\
\hline
$5$ & $(20,24)$ & $18$ & $(1,6)$ & $011101001011110011001000$ & $011100111100101110111000$ \\
\hline
$6$ & $(19,23)$ & $17$ & $(1,6)$ & $010011101110001010101100$ & $001110011110010111011100$ \\
\hline
$7$ & $(18,27)$ & $16$ & $(1,11)$ & $011001000011010001010000$ & $011101000110010000010000$ \\
\hline
$8$ & $(17,26)$ & $15$ & $(1,11)$ & $010101100010111001111000$ & $001110100011001000001000$ \\
\hline
$9$ & $(21,23)$ & $14$ & $(6,9)$ & $001110111000101010110001$ & $011000011000010011100101$ \\
\hline
$10$ & $(15,24)$ & $13$ & $(1,11)$ & $010000111010010111100110$ & $000011101000110010000010$ \\
\hline
$11$ & $(14,23)$ & $12$ & $(1,11)$ & $011000100111011100010101$ & $000001110100011001000001$ \\
\hline
$12$ & $(15,14)$ & $11$ & $(3,3)$ & $011111110100110100110010$ & $001010010111111101010110$ \\
\hline
$13$ & $(12,16)$ & $10$ & $(1,6)$ & $101111001100100001110100$ & $101110000111001111001011$ \\
\hline
$14$ & $(16,18)$ & $9$ & $(6,9)$ & $001111000010101100010111$ & $001010110000110000100111$ \\
\hline
$15$ & $(10,19)$ & $8$ & $(1,11)$ & $001101000101000001100100$ & $000100000111010001100100$ \\
\hline
$16$ & $(9,18)$ & $7$ & $(1,11)$ & $001011100111100001010110$ & $000010000011101000110010$ \\
\hline
$17$ & $(13,15)$ & $6$ & $(6,9)$ & $100010101011000100111011$ & $111001010110000110000100$ \\
\hline
$18$ & $(7,11)$ & $5$ & $(1,6)$ & $000011001000011010001010$ & $010111011100001110011110$ \\
\hline
$19$ & $(6,15)$ & $4$ & $(1,11)$ & $011101110001010101100010$ & $010000010000011101000110$ \\
\hline
$20$ & $(5,9)$ & $3$ & $(1,6)$ & $100011111010011100101000$ & $100101110111000011100111$ \\
\hline
$21$ & $(4,8)$ & $2$ & $(1,6)$ & $110010000111010010111100$ & $110010111011100001110011$ \\
\hline
$22$ & $(3,7)$ & $1$ & $(1,6)$ & $101011000100111011100010$ & $111001011101110000111001$ \\
\hline
$23$ & $(2,11)$ & $0$ & $(1,11)$ & $010100000110010000110100$ & $011001000001000001110100$ \\
\hline
\end{tabular}
}
\end{center}
\caption{Balanced generalized Pascal triangles $\PT{X_9^\infty}(i_0,j_0,24k+r)=\PT{(Z_l^\infty[24k+r],Z_r^\infty[24k+r])}$}\label{tab8}
\end{table}

Moreover, in the orbit $\orb{X_9^\infty}$ associated with the $24$-tuple $X_9=000000101000111110001101$, the existence of balanced generalized Pascal triangles for all the possible sizes can also be obtained from only $6$ positions. This result is not obtained by using Proposition~\ref{prop3} but by testing, at each position $(i_0,j_0)$ and for each remainder $r$, if the elementary blocks $V_1$ and $V_0\cup V_1$ are balanced, where $V_1=\PT{X_9^\infty}(i_0,j_0,r)$ and $V_0=\PT{X_9^\infty}(i_0,j_0,p+r)\setminus(\PT{X_9^\infty}(i_0,j_0,r)\cup\PT{X_9^\infty}(i_0+p,j_0+p,r))$. The corresponding values appear in Table~\ref{tab9}. 

\begin{table}[htbp]
\begin{center}
\begin{tabular}{|c||c||c|c|}
\hline
$r$ & $(i_0,j_0)$ & $Z_l$ & $Z_r$ \\
\hline
$1 , 7 , 15 , 23$ & $(0,9)$ & $010001000111110100111001$ & $000110010000010000011101$ \\
\hline
$4 , 5 , 12 , 13 , 21$ & $(7,22)$ & $011111110100110100110010$ & $001010010111111101010110$ \\
\hline
$3 , 6 , 14 , 19 , 22$ & $(7,4)$ & $011110011111101110000010$ & $010001001010001011100110$ \\
\hline
$2 , 10 , 18 , 20$ & $(4,15)$ & $100111011100010101011000$ & $110100010100011010010111$ \\
\hline
$0 , 8 , 16 , 22$ & $(1,2)$ & $001011100111100001010110$ & $000010000011101000110010$ \\
\hline
$1 , 3 , 9 , 11 , 17$ & $(6,7)$ & $011000100111011100010101$ & $000001110100011001000001$ \\
\hline
\end{tabular}
\end{center}
\caption{Other balanced triangles $\PT{X_9^\infty}(i_0,j_0,24k+r)=\PT{(Z_l^\infty[24k+r],Z_r^\infty[24k+r])}$}\label{tab9}
\end{table}

\section{Periodic balanced triangles modulo $m$}\label{sec:6}

The definitions of Steinhaus and generalized Pascal triangles can be extended in $\Zn{m}$ by considering the sum modulo $m$ as the local rule, instead of the sum modulo $2$. Examples of Steinhaus and generalized Pascal triangles modulo $7$ are depicted in Figure~\ref{fig10}.

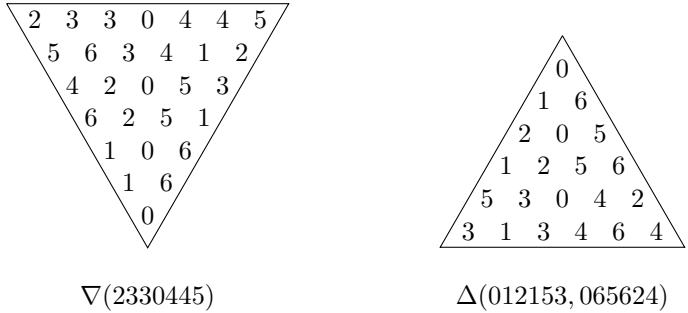
\begin{figure}[htbp]
\begin{center}
\begin{tabular}{c@{\hspace{2cm}}c}
\begin{tikzpicture}[scale=0.25]
\pgfmathparse{sqrt(3)}\let\x\pgfmathresult

\node at (0,0) {$2$};
\node at (2,0) {$3$};
\node at (4,0) {$3$};
\node at (6,0) {$0$};
\node at (8,0) {$4$};
\node at (10,0) {$4$};
\node at (12,0) {$5$};

\node at (1,-\x) {$5$};
\node at (3,-\x) {$6$};
\node at (5,-\x) {$3$};
\node at (7,-\x) {$4$};
\node at (9,-\x) {$1$};
\node at (11,-\x) {$2$};

\node at (2,-2*\x) {$4$};
\node at (4,-2*\x) {$2$};
\node at (6,-2*\x) {$0$};
\node at (8,-2*\x) {$5$};
\node at (10,-2*\x) {$3$};

\node at (3,-3*\x) {$6$};
\node at (5,-3*\x) {$2$};
\node at (7,-3*\x) {$5$};
\node at (9,-3*\x) {$1$};

\node at (4,-4*\x) {$1$};
\node at (6,-4*\x) {$0$};
\node at (8,-4*\x) {$6$};

\node at (5,-5*\x) {$1$};
\node at (7,-5*\x) {$6$};

\node at (6,-6*\x) {$0$};

\draw (-1.5,0.5*\x) -- (13.5,0.5*\x) -- (6,-7*\x) -- (-1.5,0.5*\x);
\end{tikzpicture}
&
\begin{tikzpicture}[scale=0.25]
\pgfmathparse{sqrt(3)}\let\x\pgfmathresult

\node at (0,0) {$0$};

\node at (-1,-\x) {$1$};
\node at (1,-\x) {$6$};

\node at (-2,-2*\x) {$2$};
\node at (0,-2*\x) {$0$};
\node at (2,-2*\x) {$5$};

\node at (-3,-3*\x) {$1$};
\node at (-1,-3*\x) {$2$};
\node at (1,-3*\x) {$5$};
\node at (3,-3*\x) {$6$};

\node at (-4,-4*\x) {$5$};
\node at (-2,-4*\x) {$3$};
\node at (0,-4*\x) {$0$};
\node at (2,-4*\x) {$4$};
\node at (4,-4*\x) {$2$};

\node at (-5,-5*\x) {$3$};
\node at (-3,-5*\x) {$1$};
\node at (-1,-5*\x) {$3$};
\node at (1,-5*\x) {$4$};
\node at (3,-5*\x) {$6$};
\node at (5,-5*\x) {$4$};

\draw (0,\x) -- (-6.5,-5.5*\x) -- (6.5,-5.5*\x) -- (0,\x);
\end{tikzpicture} \\[2ex]
$\ST{(2330445)}$
&
$\PT{(012153,065624)}$ \\
\end{tabular}
\caption{Steinhaus and generalized Pascal triangles in $\Zn{7}$}\label{fig10}
\end{center}
\end{figure}

The triangle $T$ is said to be {\em balanced} if its multiplicity function is constant or almost constant, i.e., if
$$
\delta\m_T := \max\left\{|\m_T(x_1)-\m_T(x_2)|\ \middle|\ x_1,x_2\in\Zn{m}\right\}\in\{0,1\}.
$$
Note that, when the triangle $T$ is of size $n$ such that the triangular number ${n+1\choose 2}$ is divisible by $m$, the triangle $T$ is balanced if $\delta\m_T = 0$, i.e., if the multiplicity function $\m_T$ is constant, equal to $\frac{1}{m}{n+1\choose 2}$. For example, the triangles in Figure~\ref{fig10}, $\ST{(2330445)}$ and $\PT{(012153,065624)}$, are balanced in $\Zn{7}$ since they contain all the elements of $\Zn{7}$ with the same multiplicity.

This generalization was introduced in \citet{Molluzzo:1978aa}, where the author posed the following problem.

\begin{MollProb}
Does there exist, for any positive integers $m$ and $n$ such that the triangular number ${n+1\choose 2}$ is divisible by $m$, a balanced Steinhaus triangle modulo $m$ of size $n$?
\end{MollProb}

This problem is still largely open. It is positively solved only for $m=2$ (Steinhaus Problem), $4$ \citep{Chappelon:2012aa}, $5$, $7$ \citep{Chappelon:2008aa} and for all $m=3^k$ with $k\in\N$ \citep{Chappelon:2008aa,Chappelon:2008ab}. It is also known \citep{Chappelon:2008aa} that there exist some values of $m$ and $n$ for which there do not exist balanced Steinhaus triangles: for $n=5$ and $m=15$ or $n=6$ and $m=21$.

In this section, some of these solutions are recalled because they involve balanced triangles that are also periodic.

First, in \citet{Chappelon:2008ab}, it was proved that, for any odd number $m$, the Steinhaus triangles generated from an arithmetic progression with an invertible common difference in $\Zn{m}$ and of length $n$ is balanced for all $n\equiv 0$ or $-1\bmod{\ord{m}{2^m}m}$, where $\ord{m}{2^m}$ is the multiplicative order of $2^m$ modulo $m$. For instance, for $(i \bmod{m})_{i\in\Z}$, the sequence of the integers modulo $m$, the Steinhaus triangle $\ST{(0,1,2,\ldots,n-1)}$ is balanced in $\Zn{m}$ for all $n\equiv 0$ or $-1\bmod{\ord{m}{2^m}m}$. In the proof of this result, it appears that the orbit generated from any arithmetic progression is periodic of period $\ord{m}{2^m}m$. This implies that all these balanced Steinhaus triangles modulo an odd number $m$ are also $\ord{m}{2^m}m$-periodic. Note that a generalization of this result in higher dimensions for balanced simplices can be found in \citet{Chappelon:2015aa} and these simplices also have periodic structure.

In \citet{Chappelon:2011aa}, the following integer sequence $S=(a_j)_{j\in\Z}$ defined by
$$
\left\{\begin{array}{ccc}
a_{3j} & = & j,\\
a_{3j+1} & = & -1-2j,\\
a_{3j+2} & = & 1+j,\\
\end{array}\right.
$$
for all $j\in\Z$, is considered. Note that this sequence is an interlacing of three arithmetic progressions. It was proved that, for every odd number $m$, the orbit of the projection of $S$ in $\Zn{m}$ contains balanced Steinhaus triangles of size $n$, for all $n$ divisible by $m$ and for all $n\equiv -1\bmod{3m}$, and balanced generalized Pascal triangles of size $n$, for all $n\equiv -1\bmod{n}$ and for all $n$ divisible by $3m$. It was proved in \citet{Chappelon:2011aa} that the orbit of this special sequence modulo $m$ is periodic of period $6m$. Thus, there exist periodic balanced triangles of these size modulo $m$ odd.

All these results lead to consider the following subproblem of the Molluzzo Problem.

\begin{prob}
Does there exist, for any positive integer $m$, infinitely many periodic balanced (Steinhaus or generalized Pascal) triangles modulo $m$?
\end{prob}

This problem is positively solved for any odd number $m$ \citep{Chappelon:2008ab,Chappelon:2011aa}, for $m=2$ (the present paper), for $m=4$ only for Steinhaus triangles \citep{Chappelon:2012aa} and for $m\in\{6,8,10\}$ \citep{Eliahou:aa}. It remains to analyse the case where $m$ is even and $m\ge 12$.

\acknowledgements
\label{sec:ack}
The author would like to thank the anonymous referees for the time spent reading this manuscript and for useful comments and remarks, which improved the presentation of the paper.
\nocite{*}
\bibliographystyle{abbrvnat}
\bibliography{biblio}
\label{sec:biblio}
\appendix
\section{The $24$-periodic orbits with balanced periods}\label{app1}

In this appendix, the orbits of representatives $X_i$, for all the $17$ elements of $\BresPO{24}=\left\{\res{X_1},\res{X_2},\ldots,\res{X_{17}}\right\}$, are given. Moreover, we have also obtained the orbits of the elements of $\BresPO{12}=\left\{\res{Y_1},\res{Y_2}\right\}$. Indeed, as already remarked, we have $\orb{Y_{1}^\infty}=\orb{X_{16}^\infty}$ and $\orb{Y_{2}^\infty}=\orb{X_{17}^\infty}$ since $X_{16}=Y_{1}^{2}$ and $X_{17}=Y_{2}^{2}$.
\begin{center}
{\tabulinesep=2mm
\begin{longtabu} to \textwidth {|X[1,c]|X[1,c]|}
\hline
 & \\*
\includegraphics[width=0.47\textwidth]{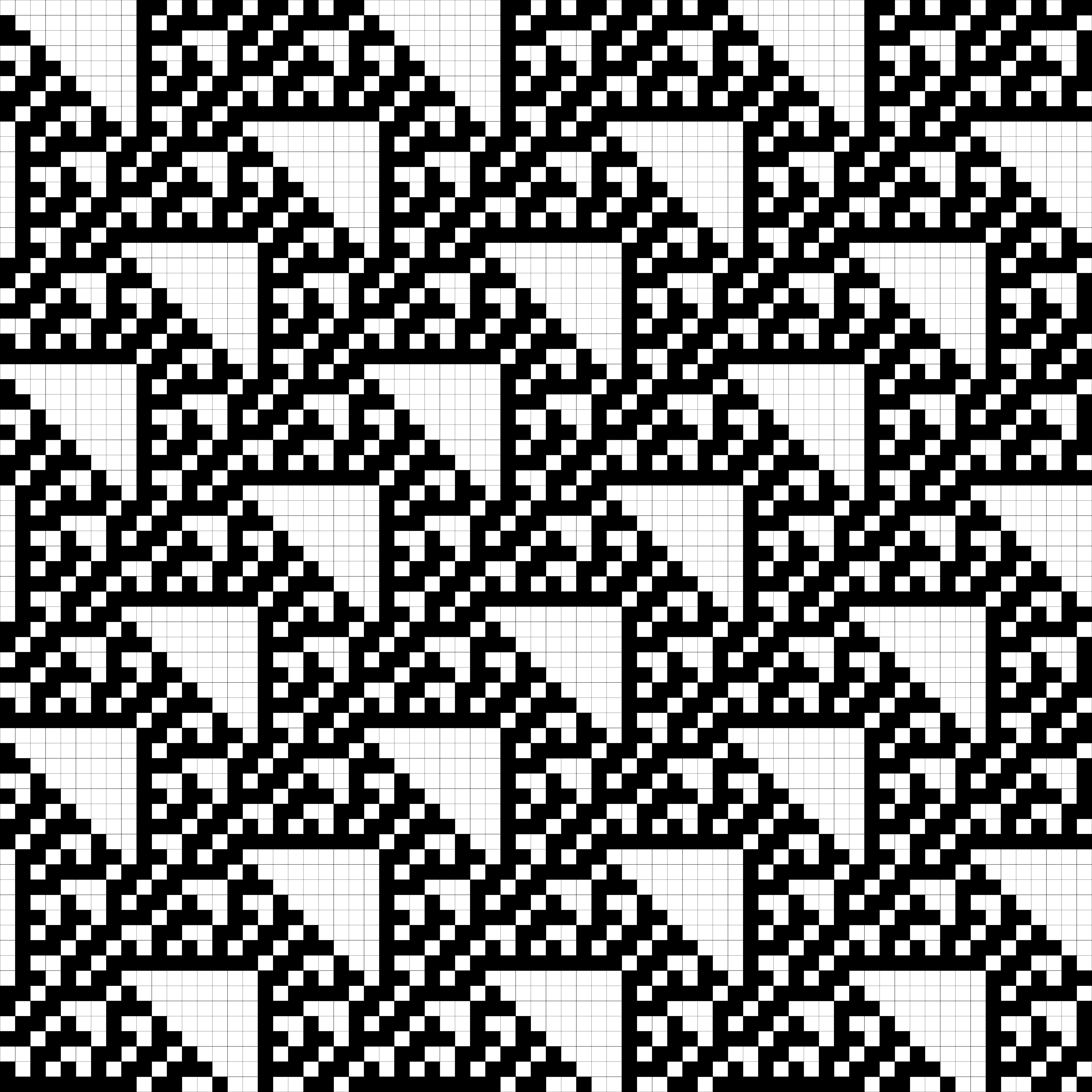} & \includegraphics[width=0.47\textwidth]{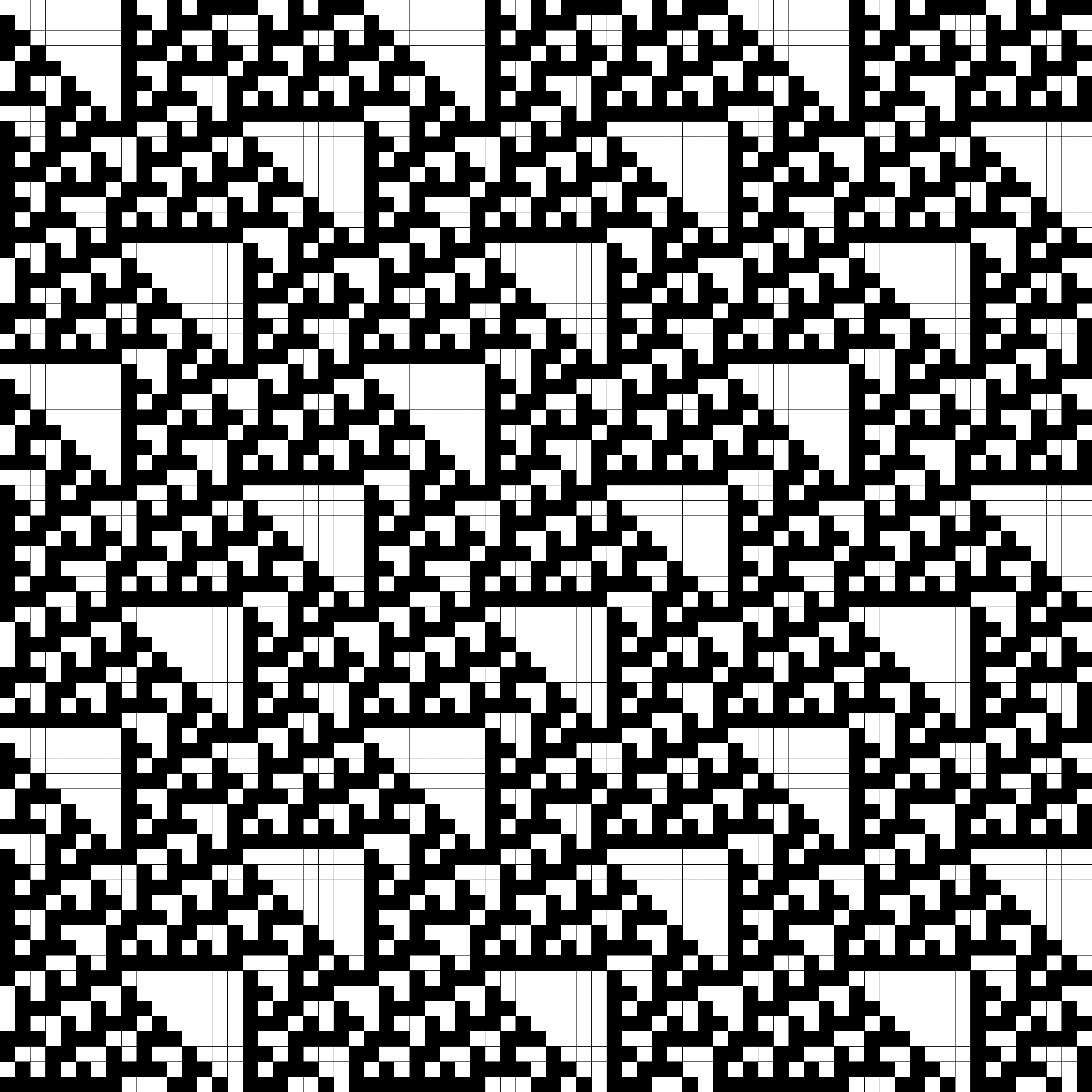} \\*
$\orb{X_1^\infty}$ & $\orb{X_2^\infty}$ \\
\hline
& \\*
\includegraphics[width=0.47\textwidth]{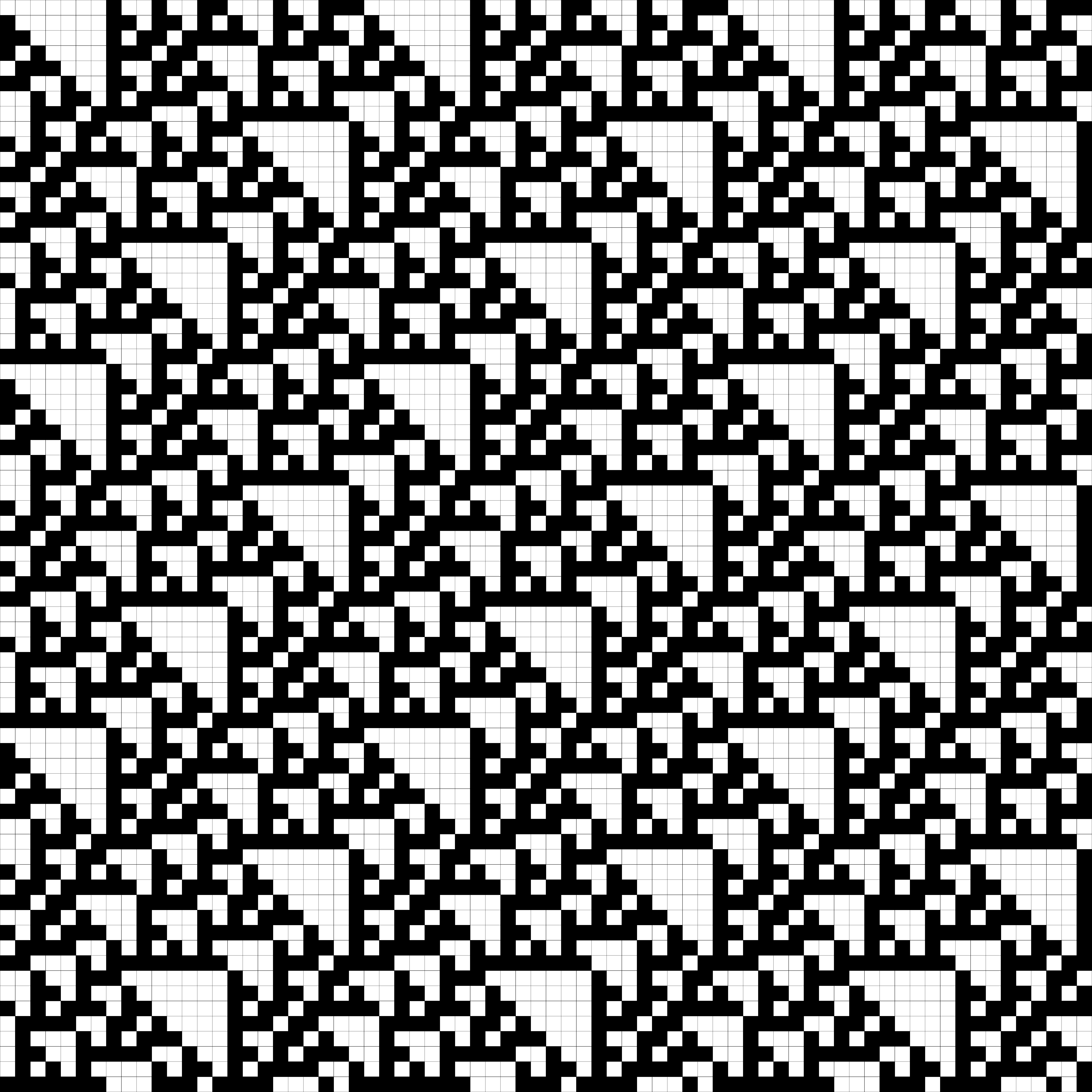} & \includegraphics[width=0.47\textwidth]{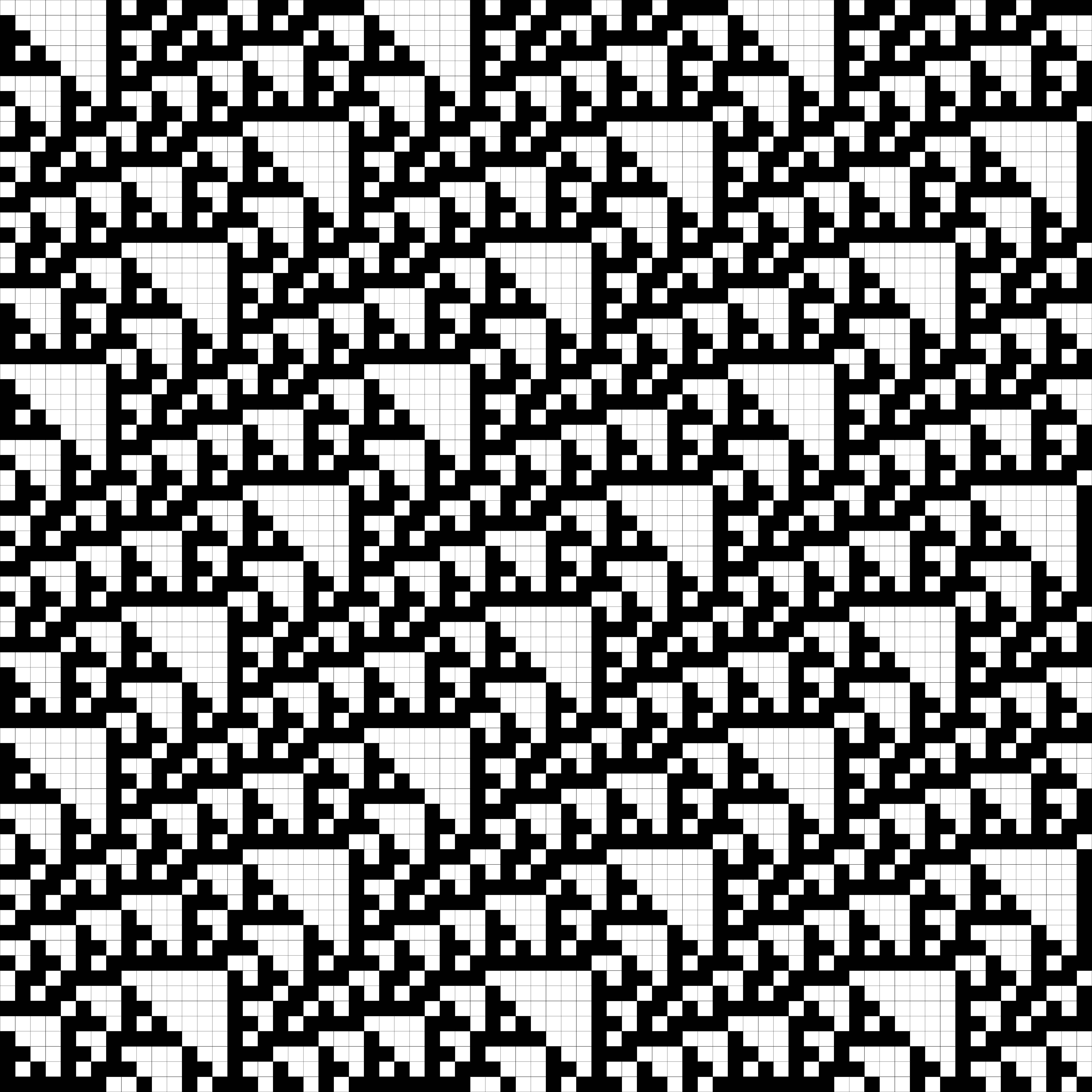} \\*
$\orb{X_3^\infty}$ & $\orb{X_4^\infty}$ \\
\hline
& \\*
\includegraphics[width=0.47\textwidth]{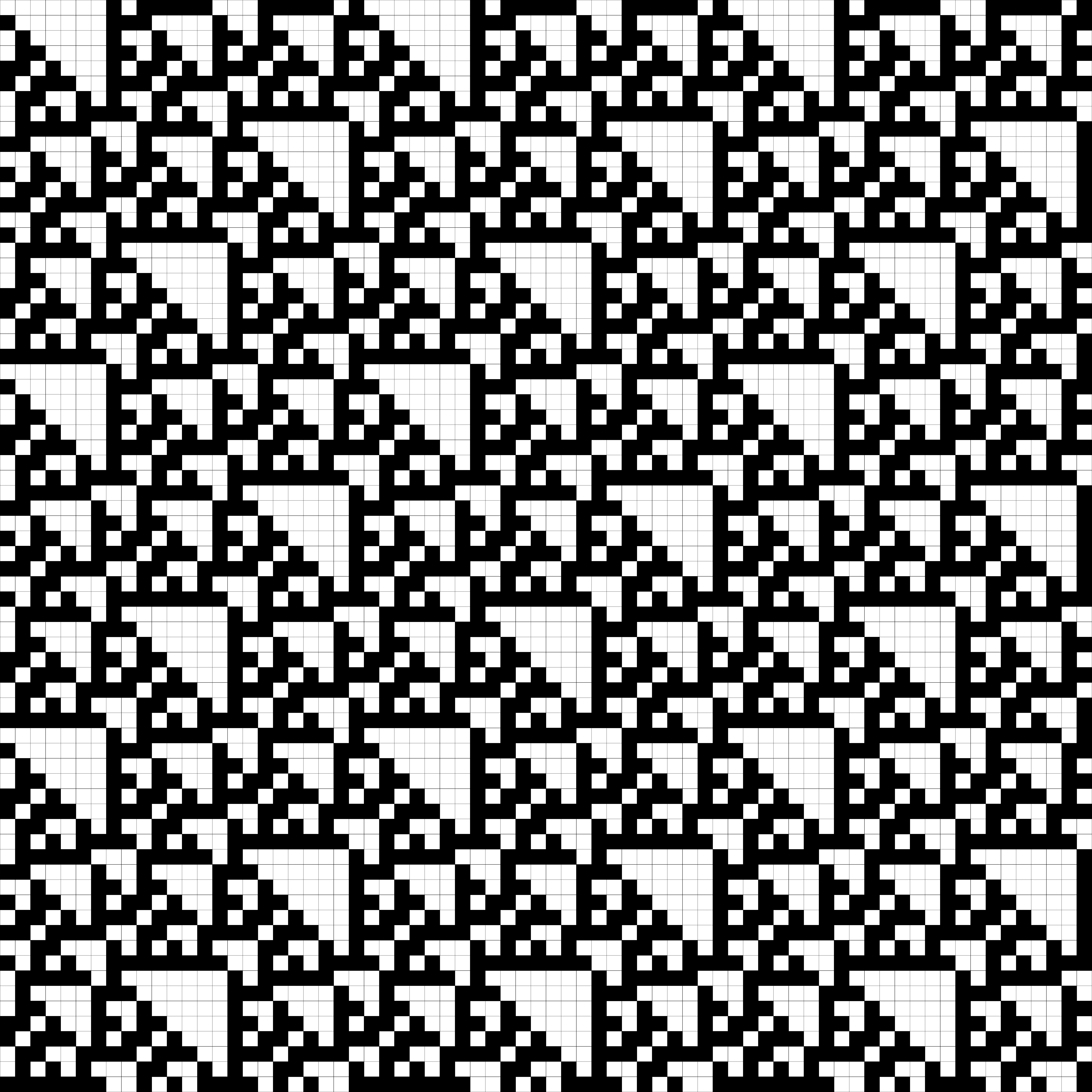} & \includegraphics[width=0.47\textwidth]{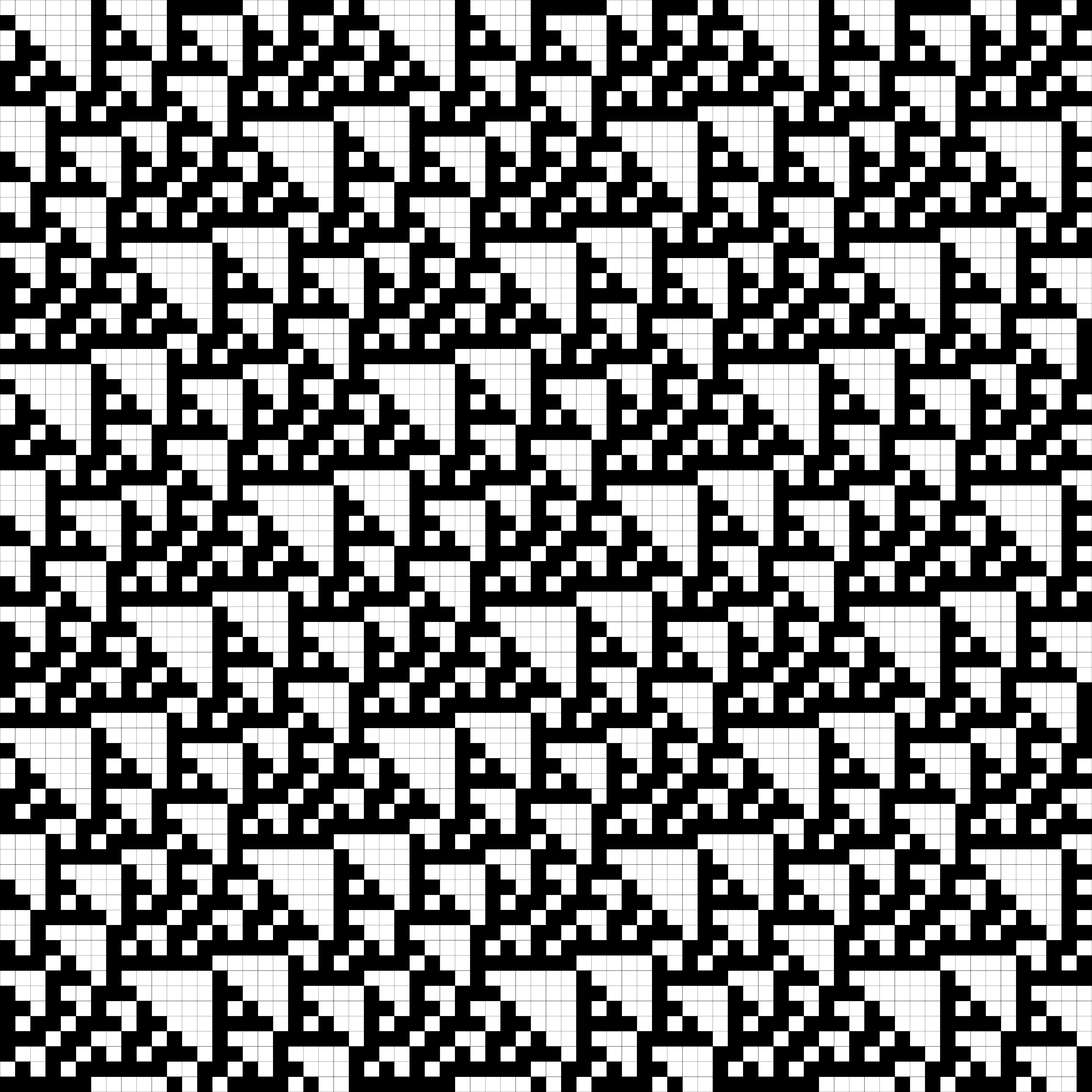} \\*
$\orb{X_5^\infty}$ & $\orb{X_6^\infty}$ \\
\hline
& \\*
\includegraphics[width=0.47\textwidth]{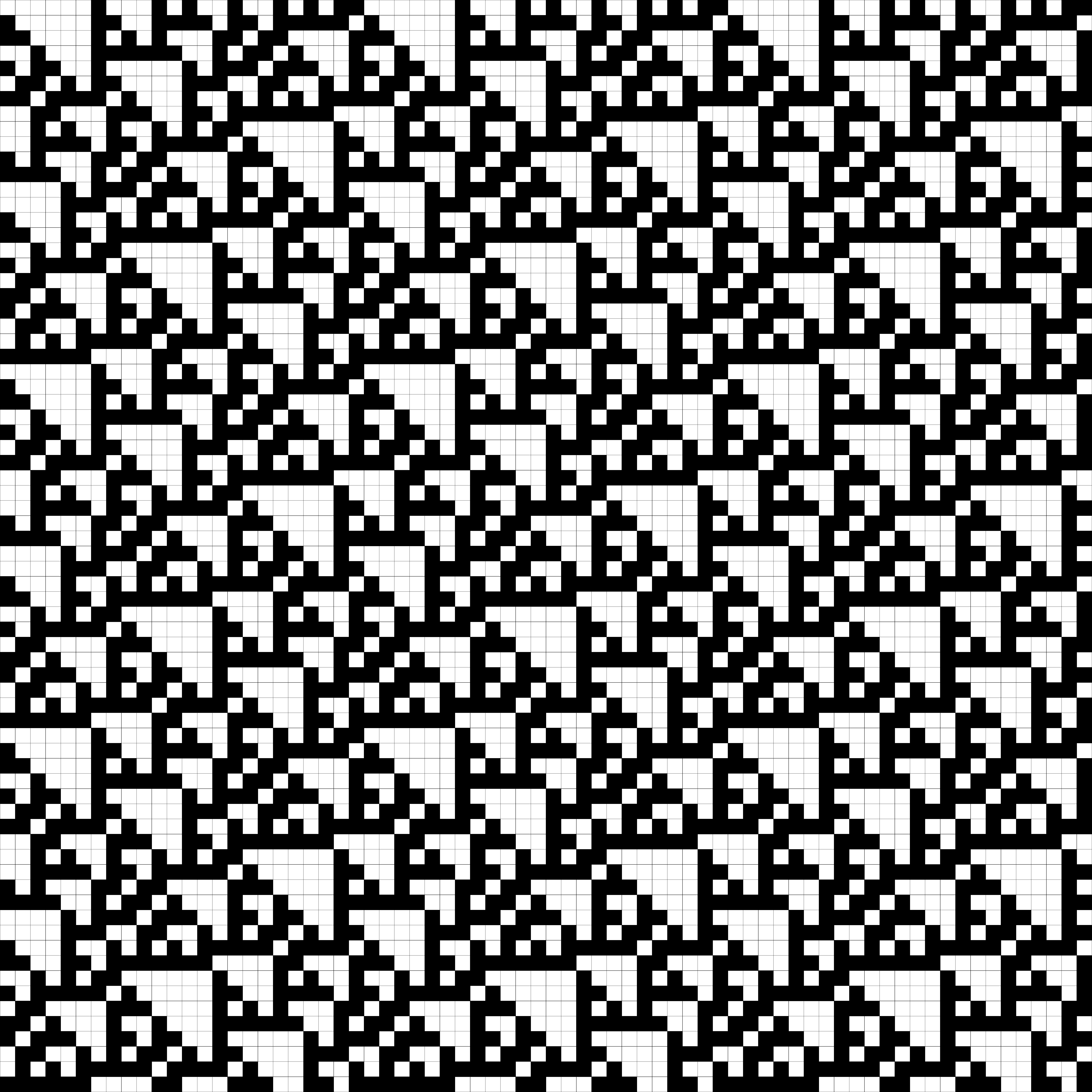} & \includegraphics[width=0.47\textwidth]{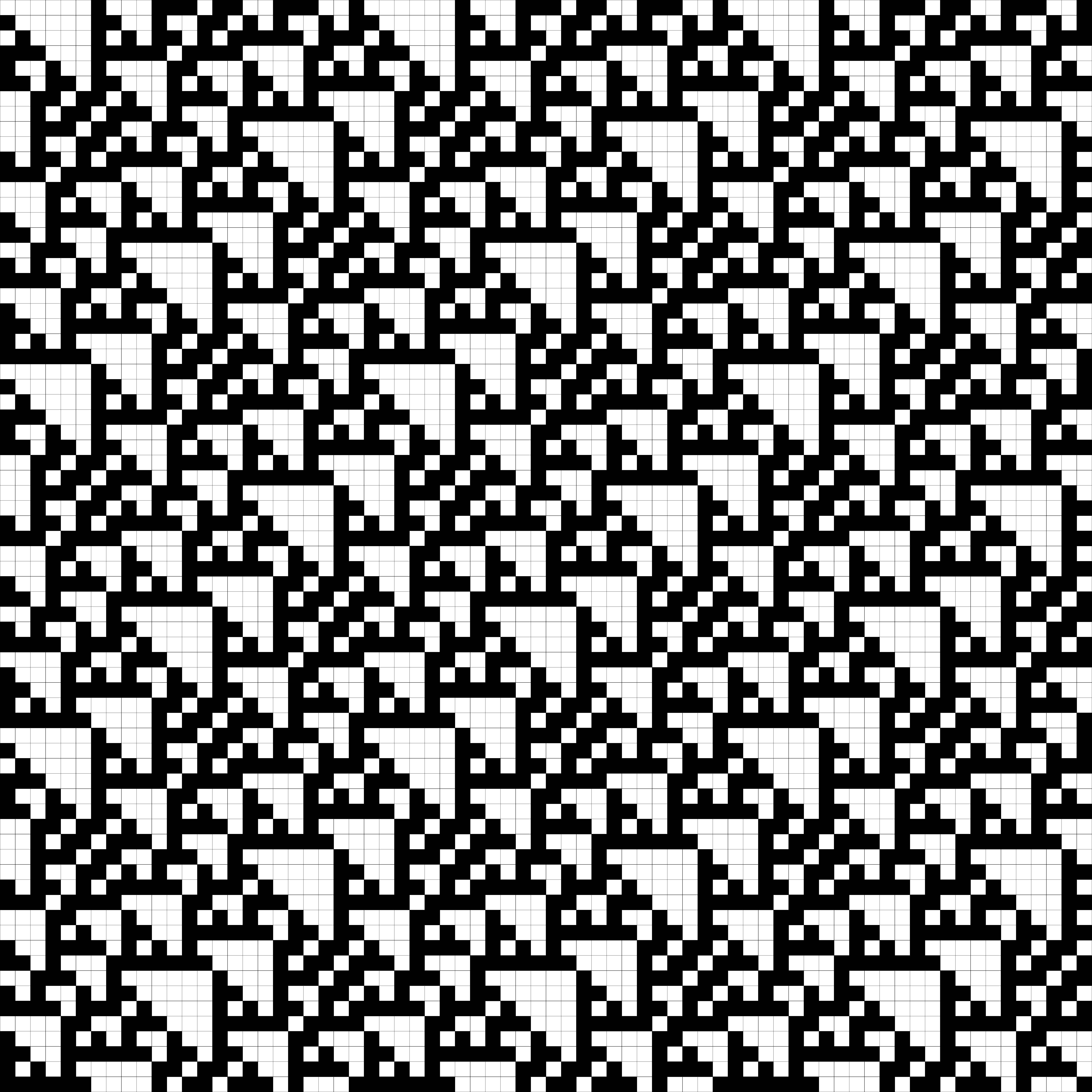} \\*
$\orb{X_7^\infty}$ & $\orb{X_8^\infty}$ \\
\hline
& \\*
\includegraphics[width=0.47\textwidth]{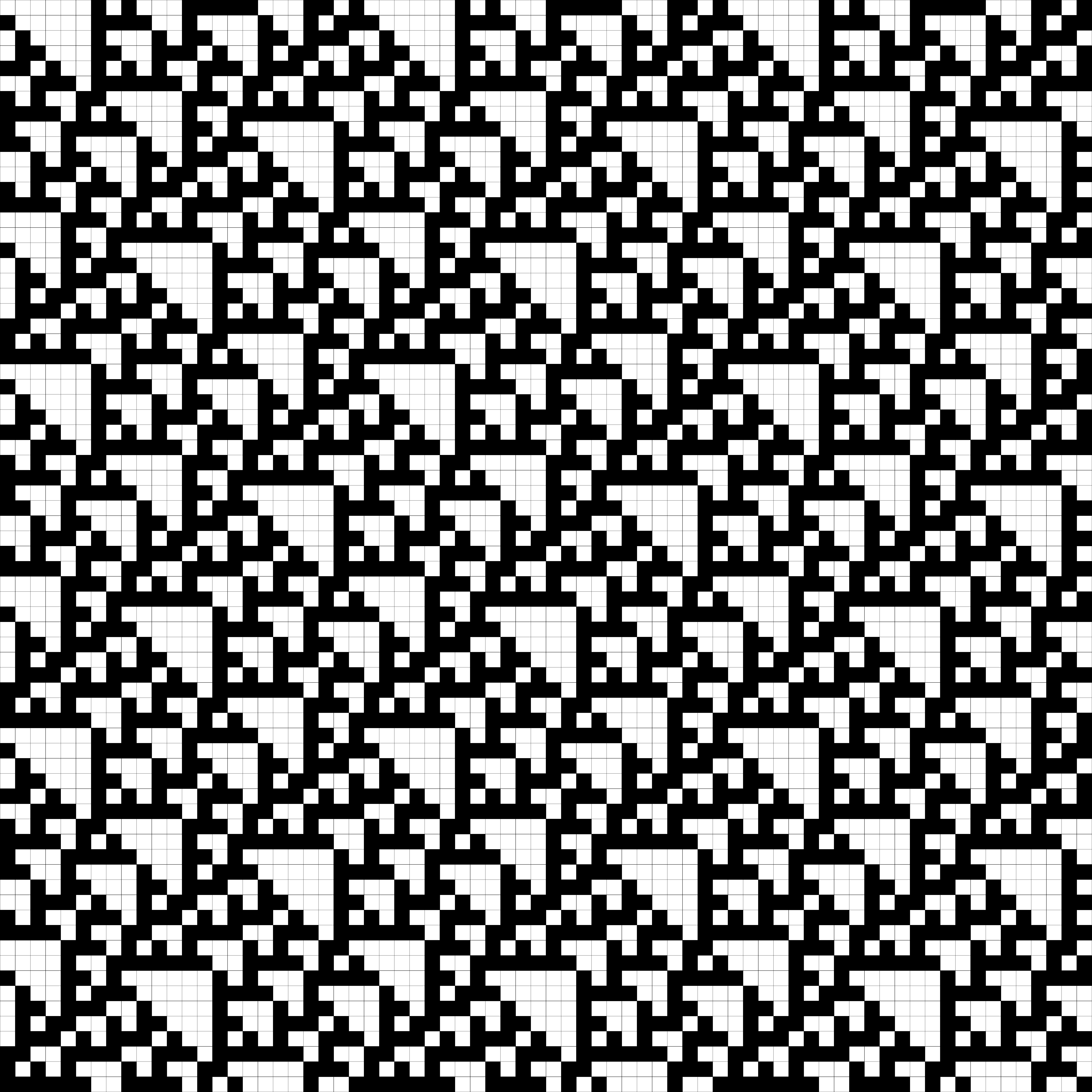} & \includegraphics[width=0.47\textwidth]{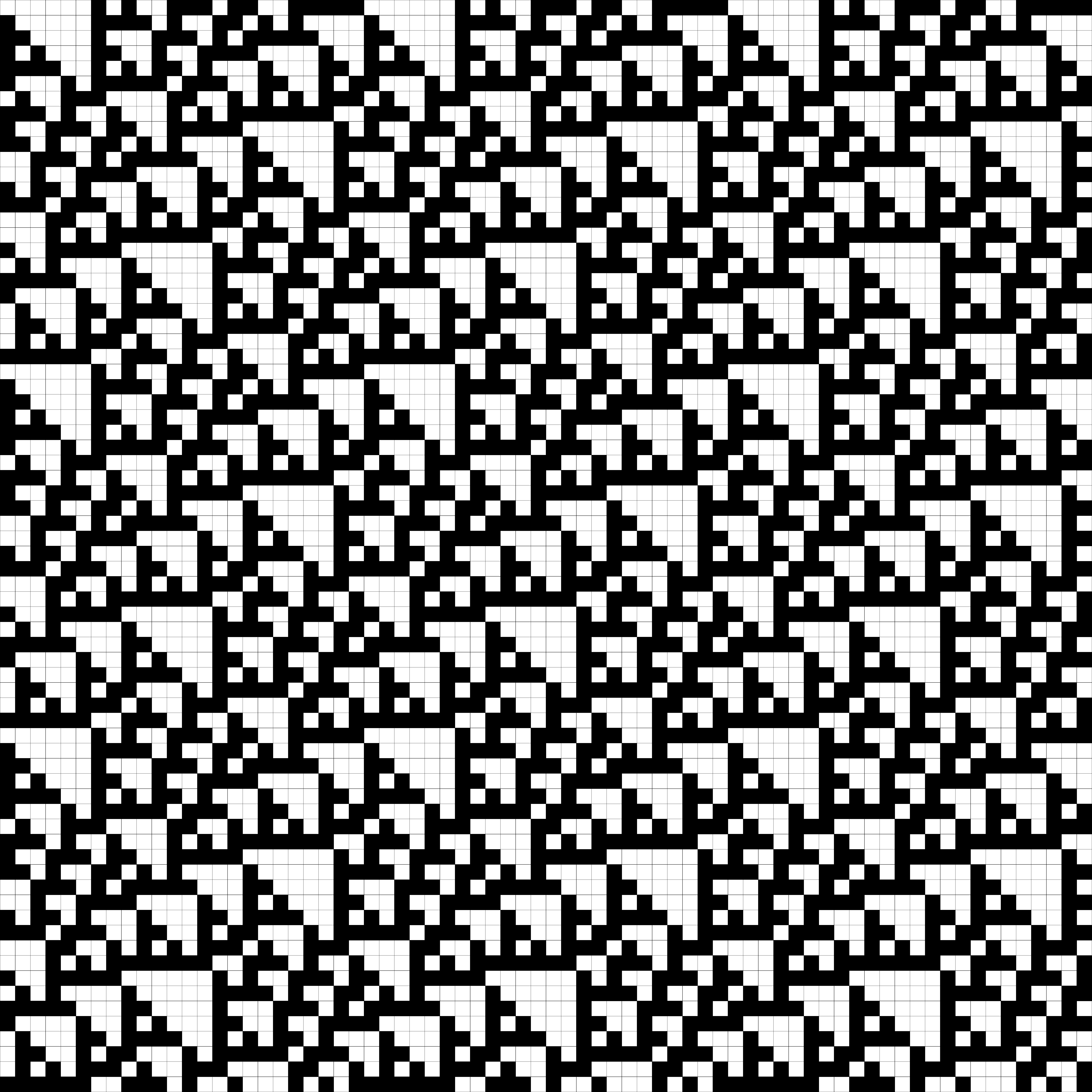} \\*
$\orb{X_9^\infty}$ & $\orb{X_{10}^\infty}$ \\
\hline
& \\*
\includegraphics[width=0.47\textwidth]{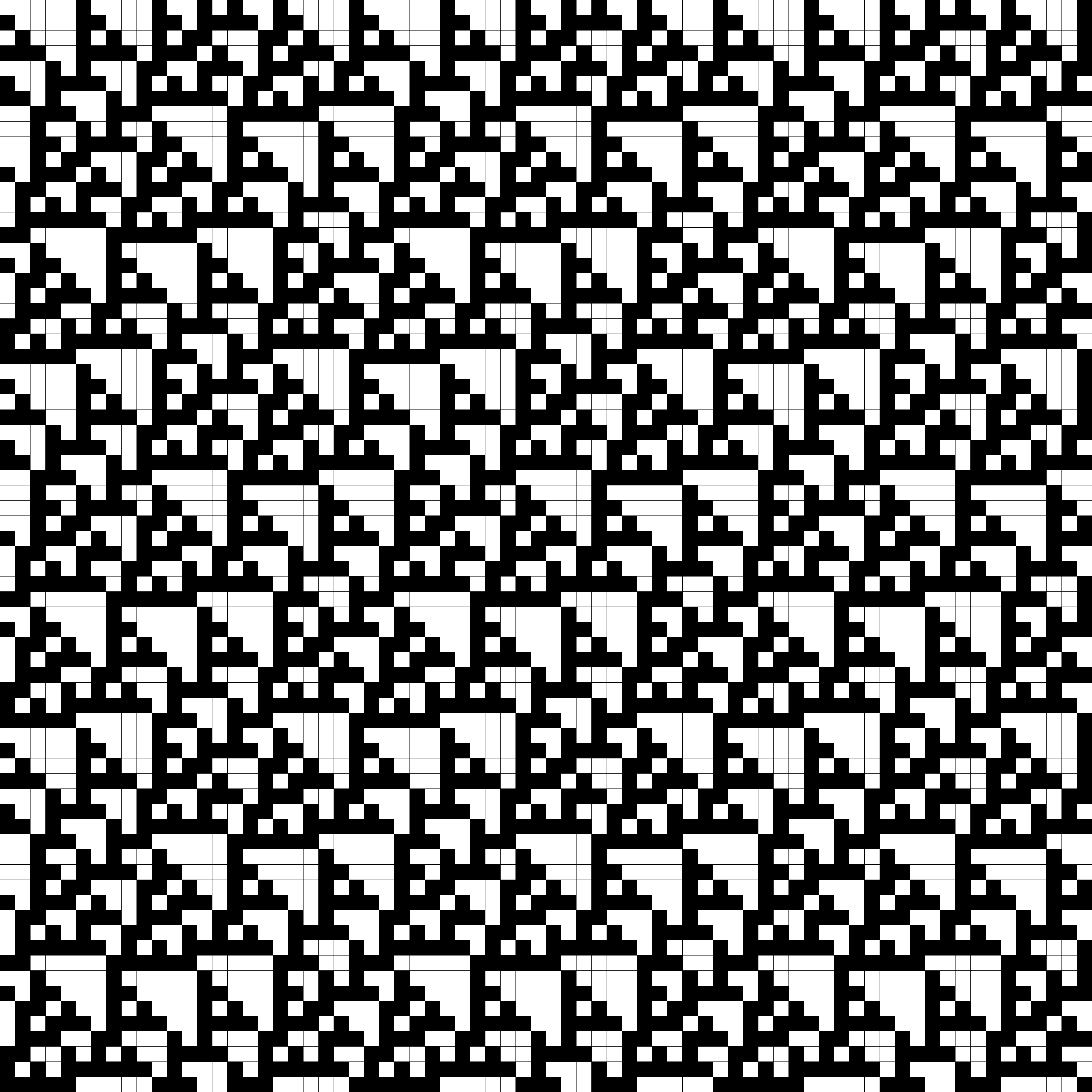} & \includegraphics[width=0.47\textwidth]{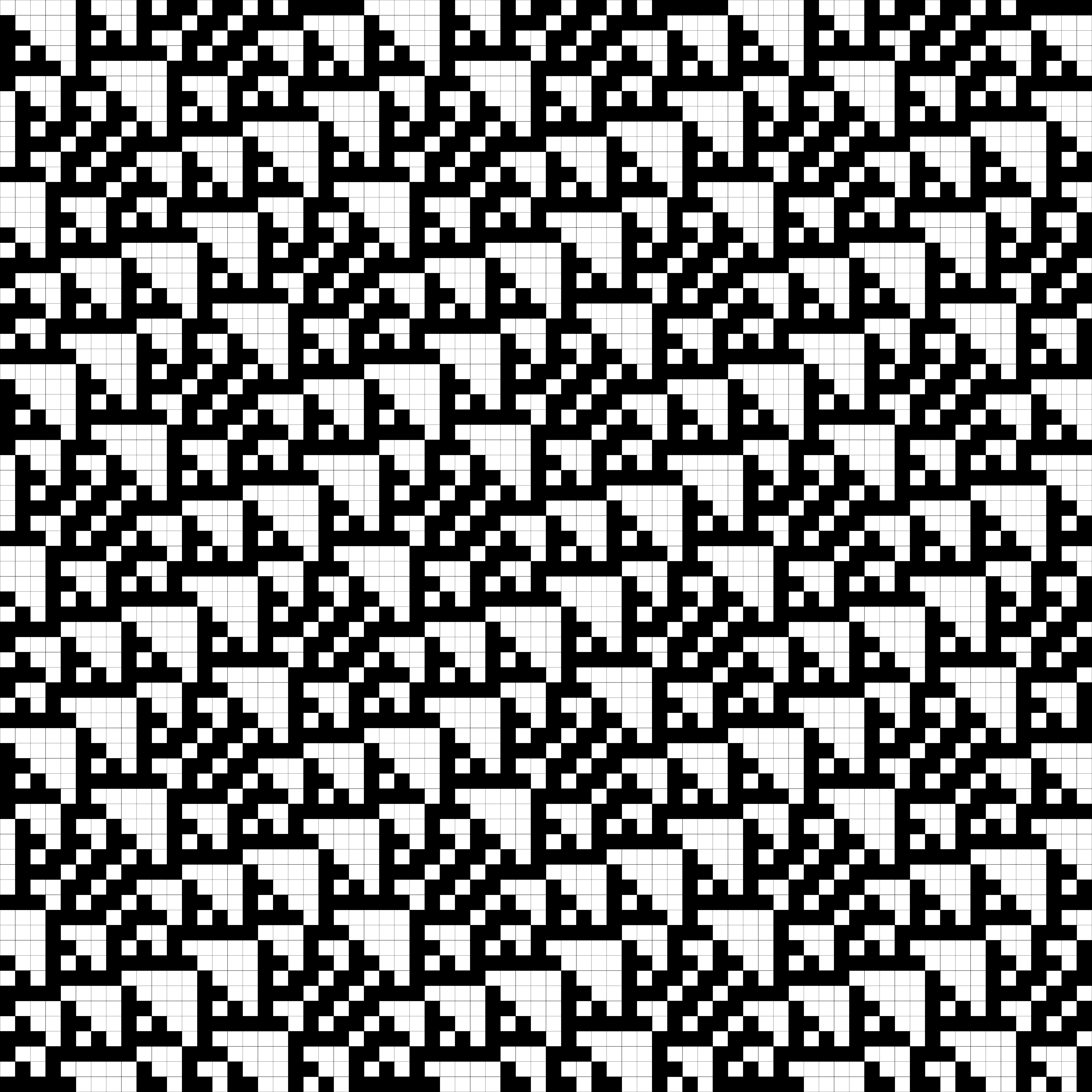} \\*
$\orb{X_{11}^\infty}$ & $\orb{X_{12}^\infty}$ \\
\hline
& \\*
\includegraphics[width=0.47\textwidth]{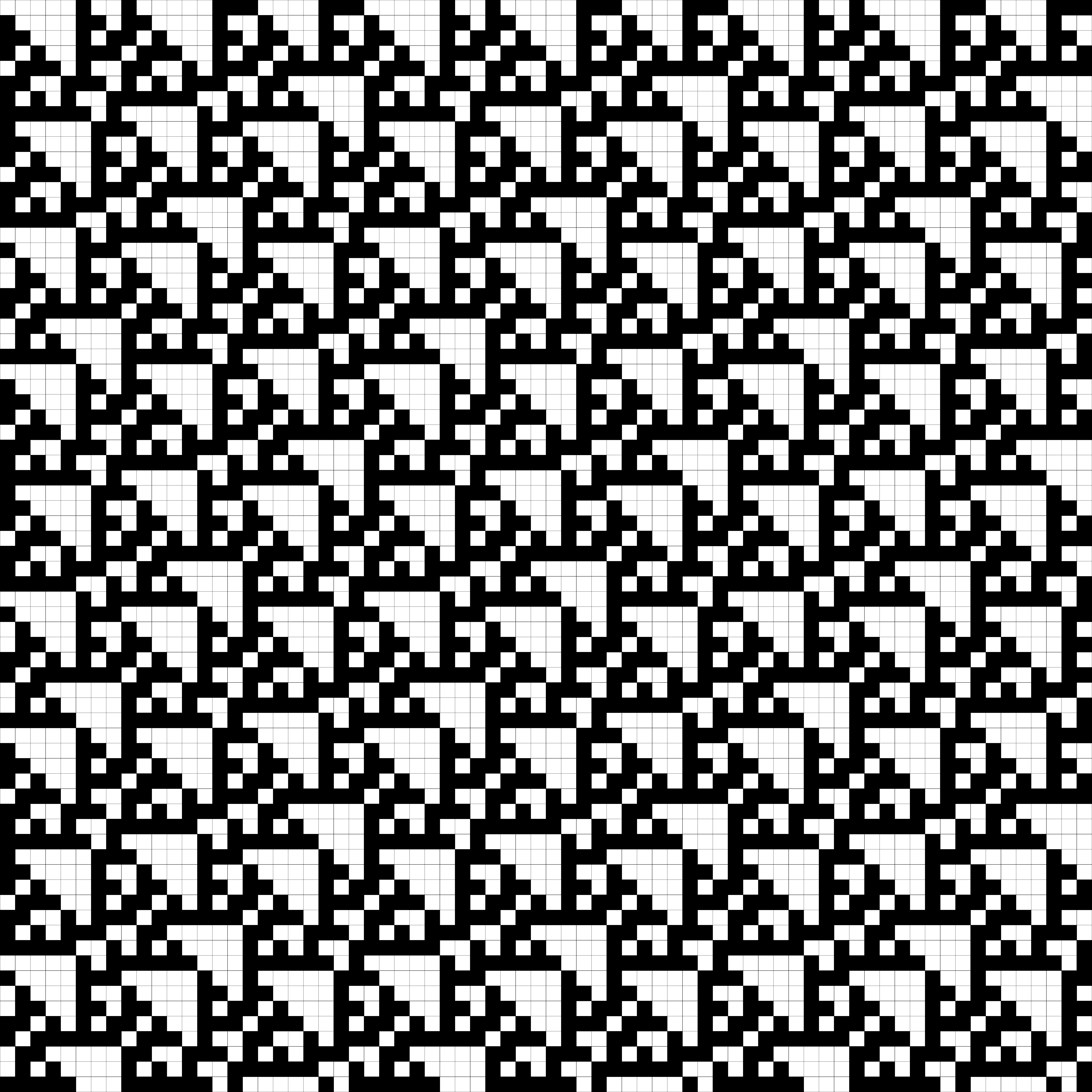} & \includegraphics[width=0.47\textwidth]{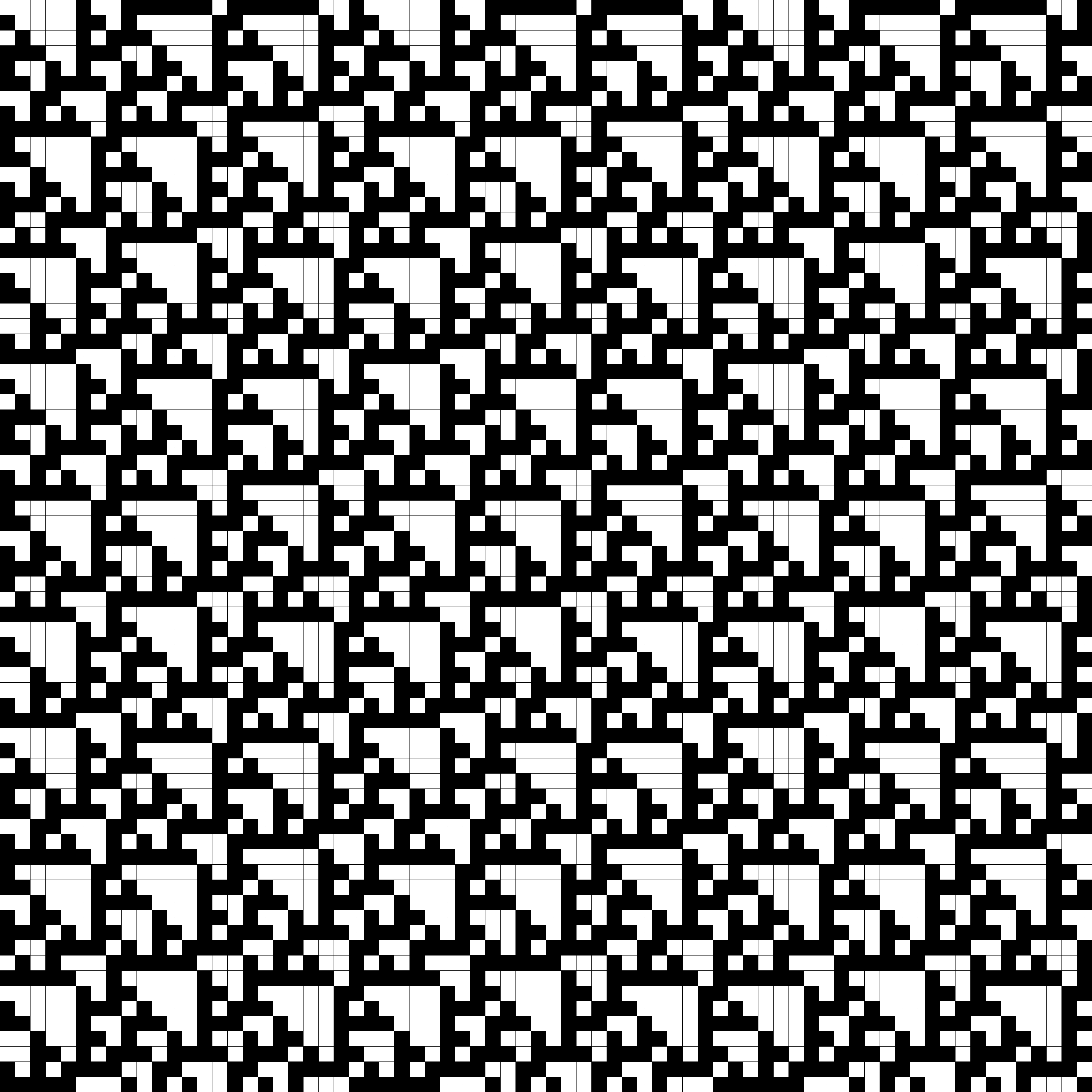} \\*
$\orb{X_{13}^\infty}$ & $\orb{X_{14}^\infty}$ \\
\hline
& \\*
\includegraphics[width=0.47\textwidth]{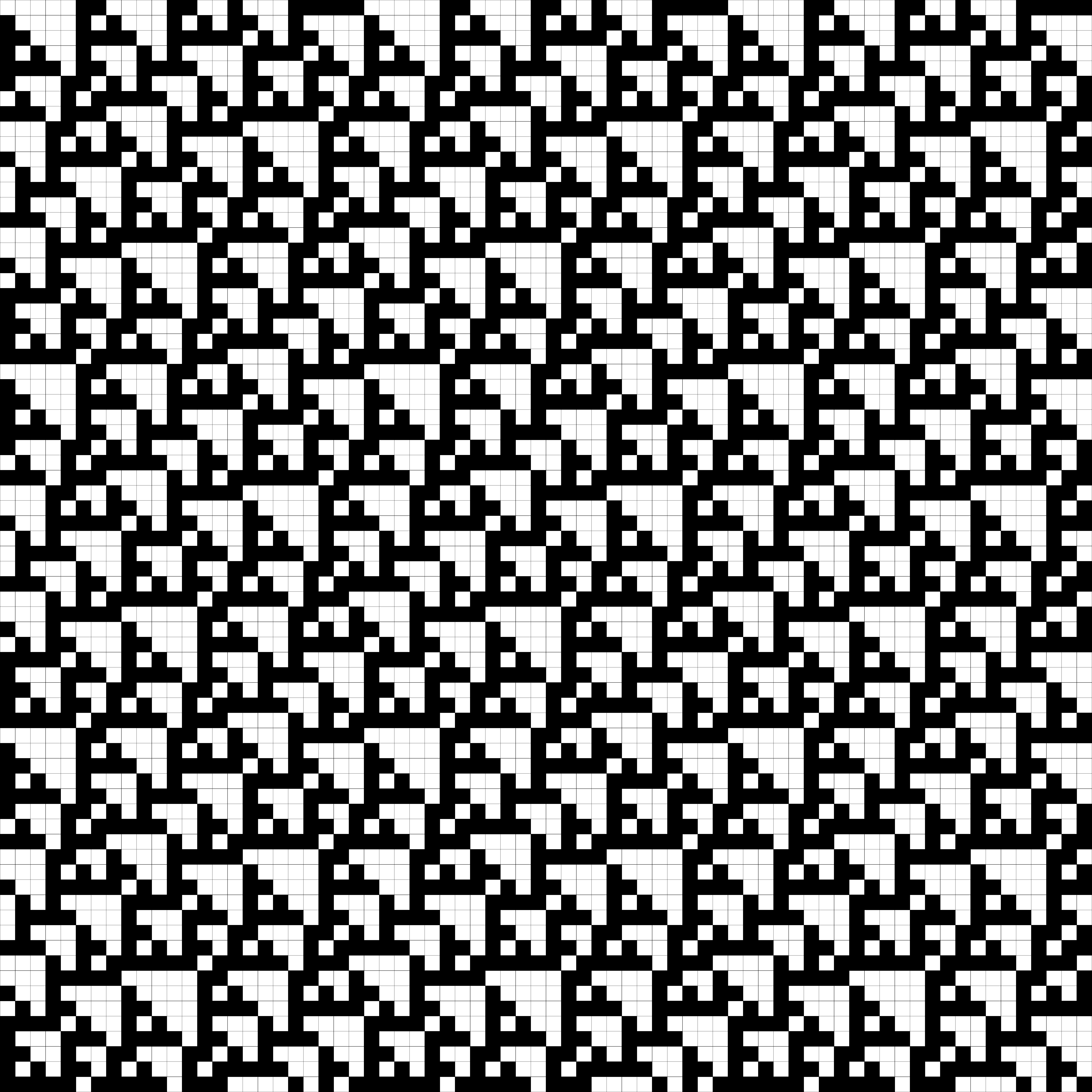} & \includegraphics[width=0.47\textwidth]{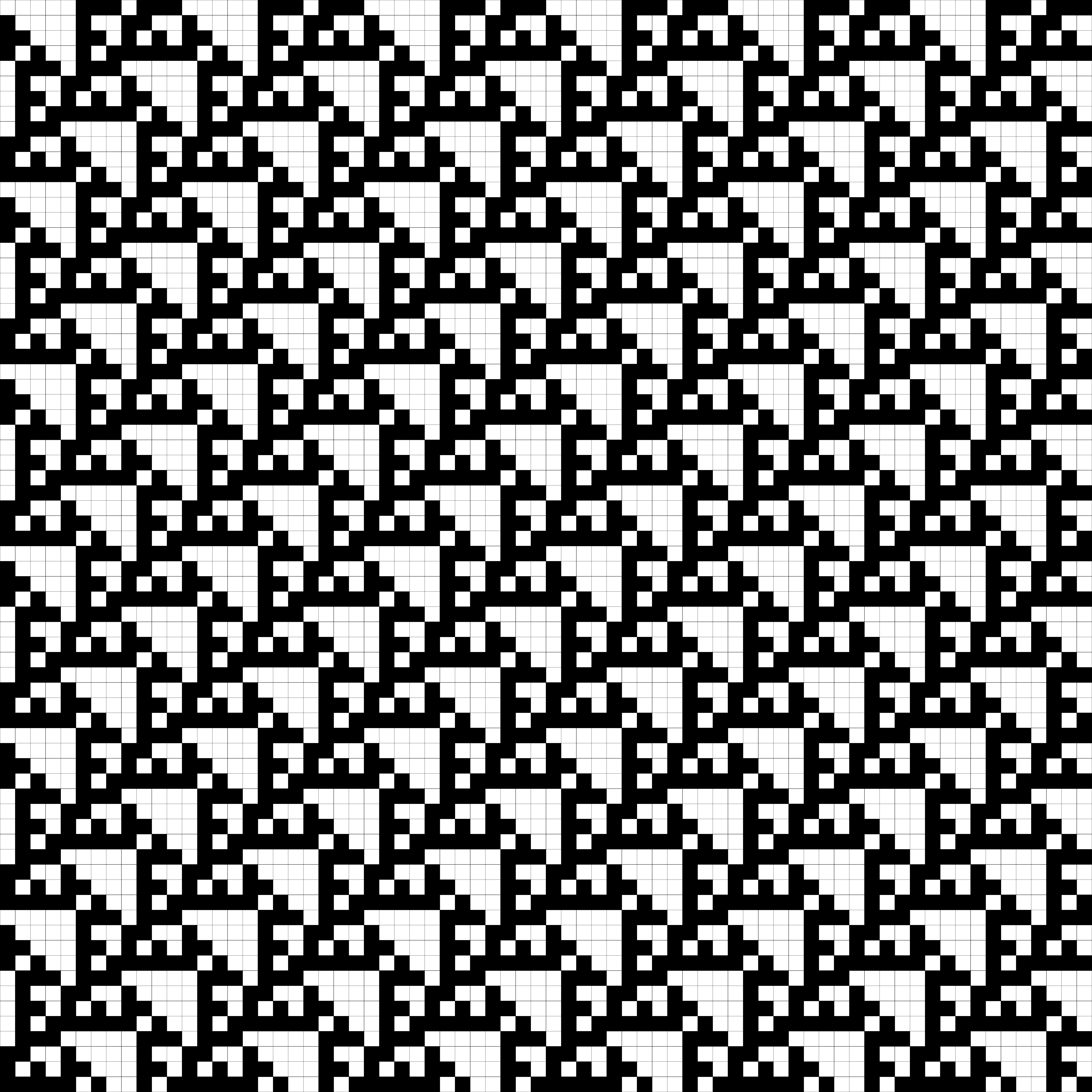} \\*
$\orb{X_{15}^\infty}$ & $\orb{X_{16}^\infty} = \orb{Y_{1}^\infty}$ \\
\hline
\multicolumn{2}{|c|}{} \\*
\multicolumn{2}{|c|}{
\includegraphics[width=0.47\textwidth]{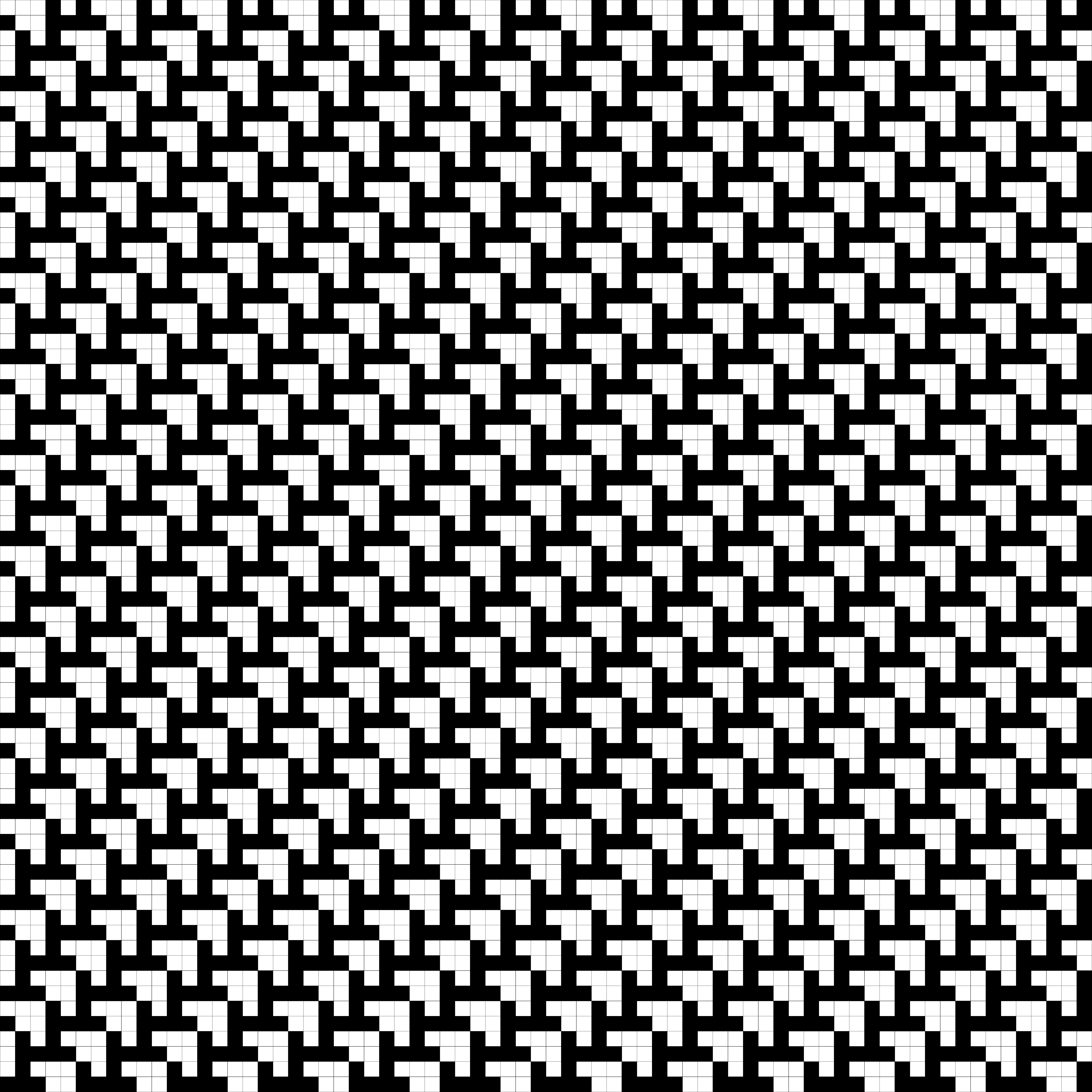}}  \\*
\multicolumn{2}{|c|}{
$\orb{X_{17}^\infty} = \orb{Y_{2}^\infty}$} \\
\hline
\end{longtabu}
}
\end{center}
\end{document}